\documentclass[11pt, reqno]{amsart}
\usepackage{amsmath}
\usepackage{amssymb}
\usepackage{amsfonts}
\usepackage[usenames]{color}
\usepackage{versions}
\usepackage{graphics}
\usepackage{url,enumerate,newlfont}
\usepackage[all,cmtip]{xy}

 \newtheorem{theorem}{Theorem}[section]
 \newtheorem{corollary}[theorem]{Corollary}
 \newtheorem{lemma}[theorem]{Lemma}
 \newtheorem{proposition}[theorem]{Proposition}

\newtheorem{definition}[theorem]{Definition}

\newtheorem{remark}[theorem]{Remark}
\newtheorem{example}[theorem]{Example}
\newtheorem{fact*}{Fact}

\newcommand\half{\tfrac 12}

\newcommand{\M}{\mathcal{M}}

\newcommand{\T}{\mathbb{T}}

\newcommand{\X}{\mathcal{X}}
\newcommand{\F}{\mathcal{F}}
\newcommand\FF{\begin{bmatrix} F_{ij}\end{bmatrix}_1^2}
\newcommand{\E}{\mathcal{E}}
\newcommand{\D}{\mathbb{D}}
\newcommand{\C}{\mathbb{C}}
\newcommand{\G}{\mathcal{G}}

\newcommand{\schur}{\mathcal{S}_2}

\newcommand{\ip}[2]{\left\langle #1, #2 \right\rangle}

\newcommand{\inv}{^{-1}}

\newcommand{\ph}{\varphi}
\renewcommand\phi{\varphi}

\newcommand\ga{\gamma}

\newcommand\la{\lambda}

\newcommand\beq{\begin{equation}}

\newcommand\eeq{\end{equation}}
\newcommand\df{\stackrel{\mathrm{def}}{=}}

\newcommand\nn{\nonumber}
\newcommand\bbm{\begin{bmatrix}}
\newcommand\ebm{\end{bmatrix}}
\newcommand\bpm{\begin{pmatrix}}
\newcommand\epm{\end{pmatrix}}

\numberwithin{equation}{section}
\DeclareMathOperator{\holf}{Hol}
\DeclareMathOperator{\trf}{tr}
\DeclareMathOperator{\idf}{id}
\DeclareMathOperator{\realf}{Re}
\DeclareMathOperator{\rankf}{rank}
\DeclareMathOperator{\spn}{span\,}

\DeclareMathOperator{\LNff}{Left\,N}
\DeclareMathOperator{\LSff}{Left\,S}
\DeclareMathOperator{\LEff}{Lower\,E}
\DeclareMathOperator{\LWff}{Lower\,W}
\DeclareMathOperator{\UWff}{Upper\,W}
\DeclareMathOperator{\UEff}{Upper\,E}
\DeclareMathOperator{\SWff}{SW}
\DeclareMathOperator{\SEff}{SE}
\DeclareMathOperator{\RSff}{Right\,S}
\DeclareMathOperator{\RNff}{Right\,N}

\newcommand{\hol}[2]{\holf\,(#1,#2)}
\newcommand{\tr}[1]{\trf\,#1}
\newcommand{\id}{\idf} 
\newcommand{\re}[1]{\realf\,#1}
\newcommand{\rank}[1]{\rankf\,(#1)}

\newcommand{\LNf}{\LNff_{\G}}
\newcommand{\LSf}{\LSff_{\G}}
\newcommand{\LEf}{\LEff_{\G}}
\newcommand{\LWf}{\LWff_{\G}}
\newcommand{\UWf}{\UWff}
\newcommand{\UEf}{\UEff}
\newcommand{\SWf}{\SWff_{\G}}
\newcommand{\SEf}{\SEff}
\newcommand{\RSf}{\RSff}
\newcommand{\RNf}{\RNff}

\newcommand{\LN}[1]{\LNf\,(#1)}
\newcommand{\LS}[1]{\LSf\,(#1)}
\newcommand{\LE}[1]{\LEf\,(#1)}
\newcommand{\LW}[1]{\LWf\,(#1)}
\newcommand{\UW}[1]{\UWf\,(#1)}
\newcommand{\UE}[1]{\UEf\,(#1)}
\newcommand{\SW}[1]{\SWf\,(#1)}
\newcommand{\SE}[1]{\SEf\,(#1)}
\newcommand{\RS}[1]{\RSf\,(#1)}
\newcommand{\RN}[1]{\RNf\,(#1)}

\newcommand{\LNfn}[1]{\LNff_{#1}}
\newcommand{\LSfn}[1]{\LSff_{#1}}
\newcommand{\LEfn}[1]{\LEff_{#1}}
\newcommand{\LWfn}[1]{\LWff_{#1}}

\newcommand{\SWfn}[1]{\SWff_{#1}}

\newcommand{\LNn}[2]{\LNff_{#1}\,(#2)}
\newcommand{\LSn}[2]{\LSff_{#1}\,(#2)}
\newcommand{\LEn}[2]{\LEff_{#1}\,(#2)}
\newcommand{\LWn}[2]{\LWff_{#1}\,(#2)}

\newcommand{\SWn}[2]{\SWff_{#1}\,(#2)}

\begin{document}

\title[A rich structure for $\Gamma$ and $\E$]{ A rich structure related to the construction of analytic matrix functions}

\author{D. C. Brown, Z. A. Lykova and N. J. Young}

\thanks{The first author  was supported by an EPSRC DTA grant. The second and third  authors were partially supported by the UK Engineering and Physical Sciences Research Council grants  EP/N03242X/1. }

\date{5th December, 2016}

\begin{abstract} 
We study certain interpolation problems for analytic $2\times 2$ matrix-valued functions on the unit disc.  We obtain a new solvability criterion for one such problem, a special case of  the $\mu$-synthesis problem from robust control theory.  
For certain domains $\mathcal X$ in $\mathbb{C}^2$ and $\mathcal{C}^3$ we describe a rich structure of interconnections between four objects: the set of analytic functions from the disc into $\mathcal X$,  the $2\times 2$ matricial Schur class, 
the Schur class of the bidisc, and the set of pairs of positive kernels on the bidisc subject to a boundedness condition.
This rich structure 
combines with the classical realisation formula and Hilbert space models in the sense of Agler to give an effective method for the construction of the required interpolating functions.
\end{abstract}

\maketitle
\section*{Contents} \label{contents}

\ref{intro}. Introduction \hfill  Page \pageref{intro}

\ref{symmetrized}. The symmetrized bidisc $\G$\hfill \pageref{symmetrized}

\ref{tetrablock}. The tetrablock $\E$ \hfill \pageref{tetrablock}

\ref{sectrealform}. A realisation formula  \hfill \pageref{sectrealform} 

\ref{kernels-maps}. 
Relations between $\cal{S}^{2\times2}$ and the set of analytic kernels on $\mathbb{D}^2$ \hfill \pageref{kernels-maps}

\hspace*{1cm} \ref{UEf}. The map $\UEf:\cal{S}^{2\times2}\to\cal{R}_{1}$  \hfill \pageref{UEf}

\hspace*{1cm} \ref{UWf}.
Procedure $UW$ and the set-valued  map $\UWf:\cal{R}_{11}\to\cal{S}^{2\times2}$
 \hfill \pageref{UWf}

\hspace*{1cm} \ref{RSf}. The map $\RSf:\cal{R}_1\to\cal{S}_2$
 \hfill \pageref{RSf}

\hspace*{1cm} \ref{RNf}.
The map $\RNf:\cal{S}_2\to\cal{R}_1$\hfill \pageref{RNf}

\ref{relsgammasect}.    Relations between $\hol{\mathbb{D}}{\Gamma}$ and other objects in the rich saltire
 \hfill \pageref{relsgammasect}

\hspace*{1cm} \ref{LNfandLSf}. The maps $\LNf:\hol{\mathbb{D}}{\Gamma}\to\cal{S}^{2\times2}$ and
$\LSf:\cal{S}^{2\times2}\to\hol{\mathbb{D}}{\Gamma}$
 \hfill \pageref{LNfandLSf}

\hspace*{1cm} \ref{LEf-G}. The map $\LEf:\hol{\mathbb{D}}{\Gamma}\to\cal{S}_2$ \hfill \pageref{LEf-G}

\hspace*{1cm} \ref{LWf-G}. The map $\LWf:\cal{S}_2^{b=c}\to \hol{\mathbb{D}}{\Gamma}$  \hfill \pageref{LWf-G}

\hspace*{1cm} \ref{SWf-G}. The map $\SWf:\cal{R}_{11}\to\hol{\mathbb{D}}{\Gamma}$
 \hfill \pageref{SWf-G}

\ref{relstetrasect}. Relations between $\hol{\mathbb{D}}{\overline{\mathcal{E}}}$ and other objects in the rich saltire
\hfill \pageref{relstetrasect}

\hspace*{1cm} \ref{LNf-E-sect}.  The map $\LNfn{\E}:\hol{\mathbb{D}}{\overline{\mathcal{E}}}\to\cal{S}^{2\times2}$ \hfill \pageref{LNf-E-sect}

\hspace*{1cm} \ref{LS-E}. The map $\LSfn{\E}:\cal{S}^{2\times2}\to\hol{\mathbb{D}}{\overline{\mathcal{E}}}$
\hfill \pageref{LS-E}

\hspace*{1cm} \ref{LEf-LWf-E}
The maps $\LEfn{\E}:\hol{\mathbb{D}}{\overline{\mathcal{E}}}\to \cal{S}_2^{\mathrm{lf}}$
and $\LWfn{\E}:\cal{S}_2^{\mathrm{lf}} \to\hol{\mathbb{D}}{\overline{\mathcal{E}}}$ \hfill \pageref{LEf-LWf-E}

\ref{crit_tetr}. A criterion for the solvability  of
the  $\mu_{\mathrm{Diag}}$-synthesis
 problem
\hfill \pageref{crit_tetr}

\ref{Proc_SW}. Construction of all interpolating functions
in $\hol{\mathbb{D}}{\overline{\mathcal{E}}}$
\hfill \pageref{Proc_SW}

 References \hfill\pageref{bibliog}

\section{Introduction}\label{intro}

Engineering provides some hard challenges for classical analysis.  In signal processing and, in particular, control theory, one often needs to construct analytic matrix-valued functions on the unit disc $\D$ or right half-plane subject to finitely many interpolation conditions and to some subtle boundedness requirements.  The resulting problems are close in spirit to the classical Nevanlinna-Pick problem, but established operator- or function-theoretic methods which succeed so elegantly for the classical problem do not seem to help for even minor variants.  For example, this is so for the {\em spectral Nevanlinna-Pick  problem} \cite{BFT1,NJY11}, which  is to construct an analytic square-matrix-valued function $F$ in $\D$
that satisfies a finite collection of interpolation conditions and the boundedness condition
\[
\sup_{\la\in\D} r(F(\la)) \leq 1 \quad \mbox{ for all }\la\in\D.
\]
This problem is a special case of the {\em $\mu$-synthesis problem} of $H^\infty$ control, which is recognised as a hard and important problem in the theory of robust control \cite{Do,DuPa}.  Even the special case of the spectral Nevanlinna-Pick problem for $2\times 2$ matrices awaits a definitive analytic theory.

A major difficulty in $\mu$-synthesis problems is to describe the analytic maps from $\D$ to a suitable domain ${\mathcal X} \subset  \C^n$ or its closure $\overline{\mathcal X}$.   In the classical theory $\mathcal X$ is a matrix ball, and the {\em realisation formula} presents the general analytic map from $\D$ to $\mathcal X$ in terms of a contractive operator on Hilbert space; this formula provides a powerful approach to a variety of interpolation problems.  In the $\mu$ variants $\mathcal X$ can be unbounded, nonconvex, inhomogeneous and non-smooth, properties which present difficulties both for an operator-theoretic approach and for standard methods in several complex variables.

In this paper we exhibit, for certain naturally arising domains ${\mathcal X}$, a rich structure  of interconnections between
four naturally arising objects of analysis in the context of $2\times 2$ analytic matrix functions on $\D$. This rich structure 
combines with the classical realisation formula and Hilbert space models in the sense of Agler to give an effective method of constructing functions in the space 
 ${\mathrm{Hol}}(\D, \overline{\X})$ of analytic maps from $\D$ to $\overline{\X}$, and thereby of obtaining solvability criteria for two cases of the $\mu$-synthesis problem.

The rich structure is summarised in the following diagram, which we call {\em the rich saltire\footnote{A heraldic term meaning an ordinary formed by a bend and a bend sinister crossing like a St. Andrew's cross (Concise Oxford Dictionary)} for the domain $\mathcal{X}$}. 
\begin{equation}\label{rich_str}
\xymatrixcolsep{7pc}
\xymatrix@R+36pt{
\quad \quad \cal{S}^{2\times 2} \quad \quad
\ar@<-3pt>[d]_{\LSfn{\X}}  
\ar[rd]^{ \SEf\quad\quad\quad\quad\quad\quad\quad} 
 \ar@<-3pt>[r]_-{\UEf}  &
  \quad \quad \cal{R}_{1}  \quad \quad 
 \ar@<-3pt>[l]_-{\UWf}  \ar@<-3pt>[d]_{\RSf} 
 \ar[ld]_{\quad\quad\quad\quad\quad\quad\quad\SWfn{\X}}\\ 
\quad \quad\hol{\mathbb{D}}{\overline{\X}} \quad \quad
 \ar@<-3pt>[u]_{\LNfn{\X}}  
 \ar@<-3pt>[r]_-{\LEfn{\X}}  & 
 \quad \quad \cal{S}_2 \quad \quad 
\ar@<-3pt>[l]_-{\LWfn{\X}} \ar@<-3pt>[u]_{\RNf} 
}
\end{equation}
The objects are defined as follows: 

$\cal{S}^{2\times 2}$ is the $2\times2$ matricial Schur class of the disc, that is, the set of analytic $2\times2$ matrix functions $F$ on $\D$ such that $\|F(\la)\| \le 1$ for all $\la \in \D$;

$\cal{S}_2$ is the Schur class of the bidisc $\mathbb{D}^2$, that is, ${\mathrm{Hol}}(\D^2, \overline{\D})$,
and 

$\cal{R}_{1}$ is the set of pairs $(N,M)$ of analytic kernels on $\mathbb{D}^2$ 
such that the kernel defined by
\[ 
(z,\lambda,w,\mu) \mapsto 1-(1-\overline{w}z)N(z,\lambda,w,\mu)-(1-\overline{\mu}\lambda)M(z,\lambda,w,\mu),
\]
for all $z,\lambda,w,\mu\in\mathbb{D}$, is positive semidefinite on $\mathbb{D}^2$ and is of rank $1$. 

The arrows in diagram \eqref{rich_str} denote mappings and correspondences that will be described in Sections 4 to 7.

In this paper we consider the rich saltire for two  domains $\X$: the symmetrised bidisc 
and the tetrablock, defined below.
Whereas $\cal{S}^{2\times2}$ and $\cal{S}_2$ are classical objects that have been much studied,
$\hol{\mathbb{D}}{\overline{\X}}$ and $\cal{R}$ have been introduced and studied within the last two decades in connection with special cases of the robust stabilisation problem.  The maps in  the upper northeast triangle of the rich saltire for a domain $\X$ do not depend on $\X$.

The {\em closed symmetrised bidisc} is defined to be the set
\[
\Gamma=\{(z+w,zw): |z|\leq1, \, |w|\leq 1\}.
\]
The {\em tetrablock} is the domain
\[
\E = \{ x \in \C^3: \quad 1-x_1 z - x_2 w + x_3 zw \neq 0 \mbox{  whenever }
|z| \le 1, |w| \le 1 \}.
\]
The closure of $\E$ is denoted by $\bar\E$.

The symmetrised bidisc arises naturally in the study of the spectral Nevanlinna-Pick problem for $2\times 2$ matrix functions.  In a similar way, the tetrablock arises from another special case of the $\mu$-synthesis problem for $2\times 2$ matrix functions \cite{NJY11}.    Define
\[
{\mathrm{Diag}} \stackrel{\mathrm{def}}{=} \left\{ \begin{bmatrix} z&0\\ 0& w\end{bmatrix}: z,w \in \C \right\}
\]
and, for a $2 \times 2$-matrix $A$,
\[
\mu_{\mathrm{Diag}} (A) = \left( \inf \{ \|X\| : X \in {\mathrm{Diag}}, 1-AX  \mbox{ is singular} \} \right)^{-1}.
\]

{\bf The  $\mu_{\mathrm{Diag}}$-synthesis
 problem}: given points $\la_1, \dots, \la_n \in \D$ and target matrices $W_1, \dots, W_n \in \C^{2 \times 2}$ one seeks an analytic $2 \times 2$-matrix-valued function $F$ such that 
\[
F(\la_j)= W_j \quad \mbox{  for } j=1,\dots,n,  \mbox{  and }
\]
\[
\mu_{\mathrm{Diag}} (F(\la)) < 1, \; \text{ for all } \; \la \in \D.
\]
This problem is  equivalent to the interpolation problem for ${\mathrm{Hol}}(\D, \E)$ studied in this paper; see \cite[Theorem 9.2]{AWY07}. 
Here  ${\mathrm{Hol}}(\D, \E)$ is the space of analytic maps from the unit disc $\D$ to $\E$.

In the case of the symmetrised bidisc a number of components of the 
rich saltire for $\Gamma$ were presented  by Agler and two of the present authors in \cite{ALY13}.
 Aspects of the rich saltire for $\Gamma$ were used in \cite[Theorem 1.1]{ALY13} to prove a solvability criterion for the $2\times 2$ spectral Nevanlinna-Pick  interpolation problem.  In this paper we give the final picture of the rich saltire for the symmetrised bidisc.

In the case of the tetrablock, with the aid of the rich saltire we obtain a solvability criterion for the $\mu_{\mathrm{Diag}}$-synthesis  problem. 
A strategy to obtain the solvability criterion is as follows.
Reduce the problem to an interpolation problem in the set of analytic functions from the disc to the tetrablock, 
induce a duality between the set ${\mathrm{Hol}}(\D, \E)$ and $\cal{S}_2$, then use Hilbert space models for $\cal{S}_2$ to obtain 
necessary and sufficient conditions for solvability. 

The main result of this paper is the existence of the rich saltire, and the principal application thereof is  the equivalence of (1) and (3) in the following assertion.

\begin{theorem} \label{NPspectral_tetr}
Let $\la_1, \dots, \la_n$ be distinct points in $\D$, let $W_1,\dots, W_n$ be  $2\times 2$ complex matrices such that
$(W_j)_{11}(W_j)_{22} \neq \det W_j$ for each $j$, and let
$(x_{1j},x_{2j},x_{3j})=((W_j)_{11},(W_j)_{22},\det W_j)$
 for each $j$.  The following three conditions are equivalent.
\begin{enumerate}[\em (1)]
\item  There exists an analytic $2\times 2$ matrix function $F$ in $\D$ such that 
\beq \label{interp-matr-main}
F(\la_j) = W_j \quad \mbox{ for }\quad j=1,\dots,n,
\eeq
and 
\beq \label{spec-cond-main}
\mu_{\mathrm{Diag}}(F(\la)) \leq 1 \quad \mbox{ for all } \quad \la \in\D.
\eeq
\item There exists a rational function $x:\mathbb{D}\to\overline{\mathcal{E}}$ such that
\beq\label{inq1.2-intr}
\qquad x(\lambda_j)=(x_{1j},x_{2j},x_{3j})\text{ for }j=1,\dots,n.
\eeq
\item  For some distinct points  $z_1,z_2, z_3$ in $\D$, there exist positive $3n$-square matrices $N=[N_{il,jk}]^{n,3}_{i,j=1,l,k=1}$ of rank at most $1$,
and $M=[M_{il,jk}]^{n,3}_{i,j=1,l,k=1}$ such that
\beq\label{inq1.5-intr}
\qquad \left[1-\overline{\frac{z_lx_{3i}-x_{1i}}{x_{2i}z_l-1}}\frac{z_kx_{3j}-x_{1j}}{x_{2j}z_k-1}\right] \geq
\left[(1-\overline{z_l}z_k)N_{il,jk}\right]+
\left[(1-\overline{\lambda_i}\lambda_j)M_{il,jk}\right].
\eeq
\end{enumerate}
\end{theorem}

This result is a part of  Theorem \ref{criterion_tetr}, which we establish in Section \ref{crit_tetr}, and  \cite[Theorem 9.2]{AWY07} (Theorem \ref{musynthequivfortetra}).
The necessary and sufficient condition for the existence of a solution of the $\mu_{\mathrm{Diag}}$-synthesis  problem for $2 \times 2$ matrix functions with $n > 2$ interpolation points is given in terms of the existence of positive $3n$-square matrices $N, M$ satisfying a certain linear matrix inequality in the data, but with the constraint that $N$ have rank $1$. 
This kind of optimization problem can be addressed with the aid of numerical algorithms (for example, \cite{boyd}), though we observe that, on account of the rank constraint, it is not a convex problem.

The paper is organized as follows.  Sections \ref{symmetrized} and  \ref{tetrablock} describe the basic properties of the symmetrized bidisc $\Gamma$ and 
the tetrablock $\E$ respectively. They also present known results on the reduction of a $2\times 2$ spectral Nevanlinna-Pick problem to an interpolation problem in
the space ${\mathrm{Hol}}(\D,\Gamma)$ of analytic functions from $\D$ to $\Gamma$,
and on the reduction of a $\mu_{\mathrm{Diag}}$-synthesis problem to an interpolation problem in
the space ${\mathrm{Hol}}(\D,\E)$ of analytic functions from $\D$ to $\E$.
In Section \ref{sectrealform}  we construct maps between the sets $\cal{S}^{2\times2}$ and  $\cal{S}_2$ using the linear fractional transformation $\cal{F}_{F(\lambda)}(z)$, $\lambda, z\in\mathbb{D}$, for $F\in\cal{S}^{2\times2}$.
Relations between $\cal{S}^{2\times2}$ and the set of analytic kernels on $\mathbb{D}^2$ are given in Section \ref{kernels-maps}. Section \ref{relsgammasect} presents the rich saltire  \eqref{rich_str_G} for  the symmetrised bidisc.
The rich saltire for the tetrablock \eqref{rich_str_E} is described in Section \ref{relstetrasect}. Here we present a duality between the space ${\mathrm{Hol}}(\D,\E)$ and a subset of the Schur class $\cal{S}_2$ of the bidisc. In Section \ref{crit_tetr} we use Hilbert space models for functions in $\cal{S}_2$ to obtain necessary and sufficient conditions for solvability of the interpolation problem in the space ${\mathrm{Hol}}(\D,\E)$.

The closed unit disc in $\C$ will be denoted by $\Delta$ and the unit circle by $\T$.  The complex conjugate transpose of a matrix $A$ will be written $A^*$.  The symbol $I$ will denote an identity operator or an identity matrix, according to context.
The $C^*$-algebra of $2 \times 2$ complex matrices will be denoted by $\M_2(\C)$. 

\section{The symmetrized bidisc $\G$}\label{symmetrized}

The {\em open} and {\em closed symmetrized bidiscs} are the subsets
\beq\label{defG}
\G=\{(z+w,zw): |z| < 1, \, |w| < 1 \}
\eeq
and 
\beq\label{defGamma}
\Gamma = \{(z+w,zw): |z| \leq 1, \, |w| \leq 1 \}
\eeq
of $\C^2$.
The sets $\G$ and $\Gamma$ are relevant to the $2\times 2$ spectral Nevanlinna-Pick problem because, for a $2\times 2$ matrix $A$, if
$r(\cdot)$ denotes the spectral radius of a matrix,
\[
r(A) < 1 \Leftrightarrow (\tr A, \det A) \in \G
\]
and
\beq\label{rGam}
r(A) \leq 1 \Leftrightarrow (\tr A, \det A) \in \Gamma.
\eeq

Accordingly, if $F$ is an analytic $2\times 2$ matrix function on $\D$ satisfying $r(F(\la)) \leq 1$ for all $\la \in\D$ then the function $(\tr F, \det F)$ belongs to the space ${\mathrm{Hol}}(\D,\Gamma)$ of analytic functions from $\D$ to $\Gamma$.  A converse statement also holds: every $\ph\in
{\mathrm{Hol}}(\D,\Gamma)$ lifts to an analytic $2\times 2$ matrix function $F$ on $\D$ such that $(\tr F,\det F)= \ph$ and consequently $r(F(\la)) \leq 1$ for all $\la \in\D$ \cite[Theorem 1.1]{AY04T}.  The $2\times 2$ spectral Nevanlinna-Pick problem can therefore be reduced to an interpolation problem in ${\mathrm{Hol}}(\D,\Gamma)$.  There is a slight complication in the case that any of the target matrices are scalar multiples of the identity matrix; for simplicity we shall exclude this case in the present paper. 

The relation \eqref{rGam} scales in an obvious way: for $\rho > 0$,
\[
r(A) \leq \rho \Leftrightarrow (\tr A, \det A) \in \rho \cdot\Gamma
\]
where 
\[
\rho\cdot(s,p) \df (\rho s,\rho^2 p) \quad \mbox{ and }\quad   \rho \cdot\Gamma\df \{\rho\cdot(s,p):(s,p)\in\Gamma\}.
\]

The following result is \cite[Proposition 3.1]{ALY13};  it is a refinement of \cite[Theorem 1.1]{AY04T}.
\begin{theorem}\label{reducetoGamma}
Let $\la_1,\dots,\la_n$ be distinct points in $\D$ and let $W_1,\dots,W_n$ be $2\times 2$ matrices, none of them a scalar multiple of the identity.  The following two statements are equivalent.
\begin{enumerate}[\em (1)]
\item There exists a rational ${2\times 2}$ matrix function $F$, analytic in $\D$, such that 
\[
F(\la_j) = W_j \quad \mbox{ for } j=1,\dots, n
\]
and
\beq\label{rineq}
 \sup_{\la\in\D}r(F(\la)) < 1;
\eeq
\item there exists a rational function $h\in{\mathrm{Hol}}(\D, \G)$ such that
\beq\label{hatnodes}
h(\la_j) = (\tr W_j, \det W_j) \quad \mbox{ for } j= 1,\dots,n,
\eeq
and $h(\D)$ is relatively compact in $\G$.
\end{enumerate}
\end{theorem}

Certain rational functions play a central role in the analysis of $\Gamma$.

\begin{definition}\label{defPhi}
The function $\Phi$ is defined for $(z, s, p) \in \mathbb{C}^3$ such that $zs\neq 2$ by
\beq\label{defPhi-formula}
\Phi(z,s,p) = \frac{2zp-s}{2-zs} = -\half s+ \frac{(p-\tfrac 14 s^2)z}{1-\half s z}.
\eeq
\end{definition}

In particular, $\Phi$ is defined and analytic on $\mathbb{D} \times
\Gamma$ (since $|s| \le 2$ when $(s, p) \in \Gamma$),
 $\Phi$ extends analytically to  $(\Delta\times \Gamma) \setminus \{(z, 2\bar z, \bar z^2): z\in\T \}$.   See \cite{AY2} for an account of how $\Phi$ arises from operator-theoretic considerations.  The $1$-parameter family $\Phi(\omega, \cdot), \; \omega\in\T,$ comprises the set of {\em magic functions} of the domain $\G$.  The notion of magic functions of a domain is explained in \cite{AY08}, 
but for this paper all we shall need is the fact that
\[
\Phi(\D \times \Gamma) \subset \Delta
\]
and a converse statement: if $w \in \C^2$ and $|\Phi(z, w)| \le 1$ for all $z \in \D$ then $w \in \Gamma$;
see for example \cite[Theorem 2.1]{AY04} (the result is also contained in  \cite[Theorem 2.2]{AY1} in a different notation).

A {\em $\Gamma$-inner function} is the analogue for ${\mathrm{Hol}}(\D,\Gamma)$ of inner functions in the Schur class.
A good understanding of rational $\Gamma$-inner functions is likely to play a part in any future solution of the finite interpolation problem for ${\mathrm{Hol}}(\D, \Gamma)$, since such a problem has a solution if and only if it has a rational $\Gamma$-inner solution  (for example, \cite[Theorem 4.2]{Cost05} or \cite[Theorem 8.1]{ALY13}). 

\begin{definition}\label{Gam-in-funct} A  {\em $\Gamma$-inner function} is an analytic function $h : \D \to \Gamma$ such that,
for almost all $\lambda \in \T$ (with respect to Lebesgue measure),
 the radial limit
\begin{equation}\label{radialG}
\lim_{r \to 1-} h(r \lambda) \; \text{ exists and belongs to} \; b \Gamma, 
\end{equation}
where $b\Gamma$ denotes the distinguished boundary of $\Gamma$.
\end{definition}
By Fatou's Theorem, the radial limit (\ref{radialG}) exists for almost all 
 $\lambda \in \T$ with respect to Lebesgue measure. 
The  distinguished boundary $b\Gamma$ of $\G$ (or $\Gamma$)  is the \v{S}ilov boundary of the algebra of continuous functions on $\Gamma$ that are analytic in $\G$.  It is the symmetrisation of the 2-torus:
$$
b\Gamma= \{ (z+w,zw): |z|=|w|=1\}.
$$
The {\em royal variety} $\mathcal{R}=\{(2z,z^2):|z| < 1\}$ plays an  important role in the theory of $\Gamma$-inner functions.

\section{The tetrablock $\E$}\label{tetrablock}

The \emph{open} and {\em closed tetrablock} are   the subsets 
\beq\label{defE}
\mathcal{E}:=\{(x_1,x_2,x_3)\in\mathbb{C}^3:1-x_1z-x_2w+x_3zw\neq0\text{ for all }z,w\in\overline{\mathbb{D}}\}
\eeq
and
\beq\label{defEbar}
\overline{\mathcal{E}}:=\{(x_1,x_2,x_3)\in\mathbb{C}^3:1-x_1z-x_2w+x_3zw\neq0\text{ for all }z,w\in \mathbb{D}\}
\eeq
of $\C^3$.

The tetrablock was introduced in  \cite{AWY07}
and is related to the $\mu_{\mathrm{Diag}}$-synthesis problem.
The following theorem was proved in \cite[Theorem 9.2]{AWY07}.

\begin{theorem}\label{musynthequivfortetra}
Let $\lambda_1,\dots,\lambda_n$ be distinct points in $\mathbb{D}$ and let 
$W_j=\begin{bmatrix}w_{11}^j&w_{12}^j\\w_{21}^j&w_{22}^j\end{bmatrix},\;$ $j=1,\dots, n$,  be $2\times2$ matrices such that $ w_{11}^j w_{22}^j \neq  \det W_j$ and  $\;\mu_{\mathrm{Diag}}(W_j) < 1$, $j=1,\dots, n$.
The following conditions are equivalent.
\begin{enumerate}[\em (1)]
\item There exists an analytic ${2\times 2}$ matrix function $F$ on $\D$, such that 
\[
F(\la_j) = W_j \quad \mbox{ for } j=1,\dots, n
\]
and
\beq\label{mu-ineq}
\sup_{\la\in\D} \mu_{\mathrm{Diag}}(F(\lambda))  < 1;
\eeq
\item there exists an analytic function $\phi \in{\mathrm{Hol}}(\D, \E)$ such that
\beq\label{x-ineq}
\phi (\la_j) = (w_{11}^j,w_{22}^j,\det{W_j}) \quad \mbox{ for } j= 1,\dots,n.
\eeq
\end{enumerate}
\end{theorem}

The following functions play a central role in the analysis of the tetrablock \cite{AWY07}. 
\begin{definition}\label{defPsiUpsilon}
The functions $\Psi,\Upsilon:\mathbb{C}^4\to\mathbb{C}$\index{$\Upsilon$} are defined for $(z,x_1,x_2,x_3) \in \C^4$ such that  $x_2z\neq1$ and $x_1z\neq1$ respectively by 
\[\Psi(z,x_1,x_2,x_3)=\frac{x_3z-x_1}{x_2z-1}\;\text{ and }\;\Upsilon(z,x_1,x_2,x_3)=\frac{x_3z-x_2}{x_1z-1}.\] 
\end{definition}
In particular 
$\Psi$ and $\Upsilon$ are defined and analytic  everywhere except when 
$x_2z=1$ and $x_1z=1$ respectively. Note that, for $x \in \C^3$ such that 
$x_1x_2=x_3$, the functions  $\Psi(\cdot,x)$ and $\Upsilon(\cdot,x)$ are constant and equal to $x_1$ and $x_2$ respectively. In this paper we will use the function $\Psi$ to define certain maps in  the rich saltire of the tetrablock.
By \cite[Theorem 2.4]{AWY07}, we have the following statement.

\begin{proposition}\label{cond_Ebar}
Let $x=(x_1,x_2,x_3)\in\mathbb{C}^3$. The following are equivalent.
\begin{enumerate}[\em (1)]
\item $ x\in \overline{\mathcal{E}}$;
\item $|\Upsilon(z,x)|\leq1$  for all $z\in\mathbb{D}$ and if $x_1x_2=x_3$ then, in addition, $|x_1|\leq 1$;
\item $ |\Psi(z,x)|\leq 1$ for all $z\in\mathbb{D}$ and if $x_1x_2=x_3$  then, in addition, $|x_2|\leq 1$;
\item $ |x_2-\overline{x_1}x_3|+|x_1x_2-x_3|\leq1-|x_1|^2$ and if $x_1x_2=x_3$ then in addition $|x_2|\leq1$;
\item $ |x_1-\overline{x_2}x_3|+|x_1x_2-x_3|\leq1-|x_2|^2$ and if $x_1x_2=x_3$ then in addition $|x_1|\leq1$;
\item $|x_1|^2+|x_2|^2-|x_3|^2+2|x_1x_2-x_3|\leq1$ and $|x_3|\leq1$;
\item there is a $ 2\times2$  matrix $A=[a_{ij}]_{i,j=1}^2$ such that $\|A\|\leq 1 $ and $x=(a_{11},a_{22},\det{A})$;
\item there is a symmetric $2\times2$  matrix 
$A=[a_{ij}]_{i,j=1}^2$ such that $\|A\|\leq 1 $ and $x=(a_{11},a_{22},\det{A})$.
\end{enumerate}
\end{proposition}

By  \cite[Theorem 2.9]{AWY07}, $\overline{\mathcal{E}}$ is polynomially convex, and so  the distinguished boundary $b\overline{\mathcal{E}}$ 
of $\overline{\mathcal{E}}$ exists and is the \u{S}ilov boundary of  the algebra $\cal{A}(\mathcal{E})$ of continuous functions
on $\overline{\mathcal{E}}$ that are analytic on $\mathcal{E}$. 
We have the following alternative descriptions of $b\mathcal{E}$ \cite[Theorem 7.1]{AWY07}.
 
\begin{theorem}\label{conditionsforbE}
Let $x=(x_1,x_2,x_3)\in\mathbb{C}^3$. The following are equivalent.
\begin{align*}
\textup{(i)}\quad& x\in b\overline{\mathcal{E}};\\
\textup{(ii)}\quad& x\in\overline{\mathcal{E}}\;\text{ and }\;|x_3|=1;\\
\textup{(iii)}\quad & x_1=\overline{x_2}x_3,\; |x_3|=1\;\text{ and }\;|x_2|\leq1;\\
\textup{(iv)}\quad& \text{either }x_1x_2\neq x_3\text{ and }\Psi(\cdot,x)\;\text{ is an automorphism of }\;\mathbb{D}\;\text{ or }x_1x_2=x_3
\text{ and }\\
\quad&|x_1|=|x_2|=|x_3|=1;\\
\textup{(v)}\quad& x\; \text{ is a peak point of }\overline{\mathcal{E}};\\
\textup{(vi)}\quad& \;\text{there is a }\;2\times2\;\text{ unitary matrix }\;U=
\begin{bmatrix} u_{ij} \end{bmatrix}_1^2 \; \text{ such that }\; x=(u_{11},u_{22},\det U);\\
\textup{(vii)}\quad& \;\text{there is a symmetric }2\times2\text{ unitary matrix } U=\begin{bmatrix} u_{ij} \end{bmatrix}_1^2 
\;\text{ such that }\;\\
 ~ & x=(u_{11},u_{22},\det U).
\end{align*}
\end{theorem}

By \cite[Corollary 7.2]{AWY07},
$b\overline{\mathcal{E}}$ is homeomorphic to $\overline{\mathbb{D}}\times\mathbb{T}$.
By a \emph{peak point} of $\overline{\mathcal{E}}$ we mean a point $p$ for which there is a function $f\in\cal{A}(\mathcal{E})$
such that $f(p)=1$ and $|f(x)|<1$ for all $x\in\overline{\mathcal{E}}\setminus\{p\}$.

\begin{definition} \label{E-in-funct}  An  {\em $\overline{\mathcal{E}}$-inner function} is an analytic function $\phi : \D \to \overline{\mathcal{E}}$ such that the radial limit
\begin{equation}\label{radialE}
\lim_{r \to 1-} \phi(r \lambda) \mbox{ exists and belongs to } b\overline{\mathcal{E}}
\end{equation}
for almost all $\lambda \in \T$.
\end{definition}

By Fatou's Theorem, the radial limit (\ref{radialE}) exists for almost all 
 $\lambda \in \T$ with respect to Lebesgue measure. 
Note that, for an $\overline{\mathcal{E}}$-inner function  $\phi= (\phi_1, \phi_2, \phi_3) : \D \to \overline{\mathcal{E}}$,  $\phi_3$ is an inner function on $\mathbb{D}$ in the classical sense.

A finite interpolation problem for ${\mathrm{Hol}}(\D,\overline{\mathcal{E}})$ has a solution if and only if it has a rational $\Gamma$-inner solution -- see  Theorem \ref{criterion_tetr}.


\section{A realisation formula}\label{sectrealform} 

In this section we construct maps between the sets $\cal{S}^{2\times2}$ and  $\cal{S}_2$.
For Hilbert spaces $H,G,U$ and $V$, an operator $P$ such that 
\[P=\begin{bmatrix} P_{11} & P_{12} \\ P_{21} & P_{22} \end{bmatrix}:H\oplus U\to G\oplus V\]
and an operator $X:V\to U$ for which $I-P_{22}X$ is invertible,
we denote by $\cal{F}_P(X)$\index{$\cal{F}_P(X)$} the linear fractional transformation 
\[
\cal{F}_P(X):=P_{11}+P_{12}X(I-P_{22}X)^{-1}P_{21}
\]
$\cal{F}_P(X)$ is an operator from $H$ to $G$. 

The following standard identity is a matter of verification.
\begin{proposition}\label{calF} 
Let $H,G,U$ and $V$ be Hilbert spaces. Let 
\[
P=\begin{bmatrix} P_{ij} \end{bmatrix}_1^2\text{ and }
Q=\begin{bmatrix} Q_{ij} \end{bmatrix}_1^2
\]
 be operators from $H\oplus U$ to $G\oplus V$.
Let $X$ and $Y$ be operators from $V$ to $U$ for which $I-P_{22}X$ and $I-Q_{22}Y$ are invertible.
Then
\begin{align*}I-\cal{F}_Q(Y)^*\cal{F}_P(X)= & \,Q^*_{21}(I-Y^*Q^*_{22})^{-1}(I-Y^*X)(I-P_{22}X)^{-1}P_{21}\\
&\quad+\begin{bmatrix}I& Q^*_{21}(I-Y^*Q^*_{22})^{-1}Y^* \end{bmatrix}(I-Q^*P)
\begin{bmatrix} I\\X(I-P_{22}X)^{-1}P_{21} \end{bmatrix}.\end{align*}
\end{proposition}

\begin{proposition}\label{calFleq1}
Let $H,G,U$ and $V$ be Hilbert spaces. Let $P=\begin{bmatrix} P_{11} & P_{12} \\ P_{21} & P_{22} \end{bmatrix}$ 
be an operator from $H\oplus U$ to $G\oplus V$
and let $X:V\to U$ be an operator for which $I-P_{22}X$ is invertible.
Then if $\|X\|\leq1$ and $\|P\|\leq1$ 
we have $\|\cal{F}_{P}(X)\|\leq1$. 
\end{proposition}

\begin{proof}
By Proposition \ref{calF},
\begin{align*}I-\cal{F}_P(X)^*\cal{F}_P(X)= & \,P^*_{21}(I-X^*P^*_{22})^{-1}(I-X^*X)(I-P_{22}X)^{-1}P_{21}\\
&\quad+\begin{bmatrix} I&P^*_{21}(I-X^*P^*_{22})^{-1}X^* \end{bmatrix}(I-P^*P)
\begin{bmatrix} I\\X(I-P_{22}X)^{-1}P_{21} \end{bmatrix}.\end{align*}
Let $A=(I-P_{22}X)^{-1}P_{21}:H\to V$ and
\[B=\begin{bmatrix} I\\X(I-P_{22}X)^{-1}P_{21} \end{bmatrix}=\begin{bmatrix} I\\XA \end{bmatrix}:H\to H\oplus U.\]
Then \[I-\cal{F}_P(X)^*\cal{F}_P(X)=A^*(I-X^*X)A+B^*(I-P^*P)B.\]
By assumption, $\|X\|\leq1$ and $\|P\|\leq1$, and so
 \[I-X^*X\geq0\text{ and }I-P^*P\geq0.\]
 Hence, by
\cite[Theorem 4.2.2 (iii)]{KaRi1},
$I-\cal{F}_P(X)^*\cal{F}_P(X)\geq0$. Therefore, $\|\cal{F}_P(X)\|\leq1$, as required.
\end{proof}

Recall that $\cal{S}^{2\times 2}$ is the set of
analytic maps $F:\mathbb{D}\to\cal{M}_{2}(\mathbb{C})$ such that $\|F(\lambda)\|\leq1$  for every $\lambda\in\mathbb{D}$.
For each $F=\FF
\in\cal{S}^{2\times2}$, 
we define functions $\gamma$ and $\eta$ by
\begin{equation}\label{gamma-eta-2}
\gamma(\lambda,z)=(1-F_{22}(\lambda)z)^{-1}F_{21}(\lambda)\index{$\gamma(\lambda,z)$}\;
\text{ and } \;\eta(\lambda,z)=
\begin{bmatrix}1\\z(1-F_{22}(\lambda)z)^{-1}F_{21}(\lambda) \end{bmatrix}=\begin{bmatrix}1\\z\gamma(\lambda,z) \end{bmatrix}
\end{equation}
for all $\lambda\in\mathbb{D}$ and $z\in\mathbb{C}$ such that $1-F_{22}(\lambda)z\neq0$.

\begin{proposition} \label{1-calfcalf}
Let $F=\FF\in\cal{S}^{2\times2}.$ 
Then 
\[1-\cal{F}_{F(\mu)}(w)^*\cal{F}_{F(\lambda)}(z)=\overline{\gamma(\mu,w)}(1-\overline{w}z)\gamma(\lambda,z)+
\eta(\mu,w)^*(I-F(\mu)^*F(\lambda))\eta(\lambda,z)\]
for all $\mu,\lambda\in\mathbb{D}$ and $w,z\in\mathbb{C}$ such that
$1-F_{22}(\mu)w\neq0$ and $1-F_{22}(\lambda)z\neq0$.
Moreover, $|\cal{F}_{F(\lambda)}(z)|\leq1$ for all $\lambda\in\mathbb{D}$ and $z\in\overline{\mathbb{D}}$ such that 
$1-F_{22}(\lambda)z\neq0$. 
\end{proposition}

\begin{proof}
Let $H=G=U=V=\mathbb{C}$, $P=F(\lambda)$, $Q=F(\mu)$, $X=z$ and $Y=w$
in Proposition \ref{calF}. Then 
\begin{align*}
1-\cal{F}_{F(\mu)}(w)^*&\cal{F}_{F(\lambda)}(z)\\
=&\,\overline{F_{21}(\mu)}(1-\overline{w}\overline{F_{22}(\mu)})^{-1}(1-\overline{w}z)(1-F_{22}(\lambda)z)^{-1}
F_{21}(\lambda)\\
&\,+\begin{bmatrix} 1&\overline{F_{21}(\mu)}(1-\overline{w}\overline{F_{22}(\mu)})^{-1}\overline{w} 
\end{bmatrix}(I-F(\mu)^*F(\lambda))\begin{bmatrix} 1\\z(1-F_{22}(\lambda)z)^{-1}F_{21}(\lambda) 
\end{bmatrix}\\
=&\,\overline{\gamma(\mu,w)}(1-\overline{w}z)\gamma(\lambda,z)+\eta(\mu,w)^*(I-F(\mu)^*F(\lambda))\eta(\lambda,z)
\end{align*}
for all $\mu,\lambda\in\mathbb{D}$ and $w,z\in\mathbb{C}$ such that
$1-F_{22}(\mu)w\neq0$ and $1-F_{22}(\lambda)z\neq0$.
Since $F\in\cal{S}^{2\times2}$ we have $\|F(\lambda)\|\leq1$ for all $\lambda\in\mathbb{D}$.
Hence, by Proposition \ref{calFleq1}, 
$|\cal{F}_{F(\lambda)}(z)|\leq1$ for all $\lambda\in\mathbb{D}$ and $z\in\overline{\mathbb{D}}$ such that $1-F_{11}(\lambda)z\neq0$,
as required. 
\end{proof}

\begin{remark}\label{calFholo} {\em
If we take $U=V=\mathbb{C}^n$ and $X=\lambda$, $\lambda \in \D$, in Proposition \ref{calFleq1} then we deduce that 
$$\cal{F}_P(\lambda)=P_{11}+P_{12}\lambda(I-P_{22}\lambda)^{-1}P_{21}$$
 is analytic on $\mathbb{D}$, since 
$I-P_{22}\lambda$ is invertible for all $\lambda\in\mathbb{D}$.

Thus, for $F=\FF\in\cal{S}^{2\times2}$, the linear fractional transformation 
$\cal{F}_{F(\lambda)}(z)$ is given by 
\[\cal{F}_{F(\lambda)}(z):=F_{11}(\lambda)+F_{12}(\lambda)z(1-F_{22}(\lambda)z)^{-1}F_{21}(\lambda),\]
where $\lambda\in\mathbb{D}$ and $z\in\mathbb{C}$ is such that $1-F_{22}(\lambda)z\neq0$.
}
\end{remark}

\begin{definition}\label{SEf_all}
The map
$$\SEf:\cal{S}^{2\times2}\to\cal{S}_2$$
is given by
\[\SE{F}(z,\lambda):=-\cal{F}_{F(\lambda)}(z), \; z,\lambda\in\mathbb{D}.\] 
\end{definition}

\begin{proposition}
The map $\SEf$ is well defined.
\end{proposition}

\begin{proof}
Let $F\in\cal{S}^{2\times2}$. By Remark \ref{calFholo}, 
$\SE{F}$ is analytic on $\mathbb{D}^2$. 
By Proposition \ref{1-calfcalf}, for all $z\in\mathbb{D}$, 
\[|\cal{F}_{F(\lambda)}(z)|\leq 1\text{ for all }\lambda\in\mathbb{D}.\] Hence $\SE{F}(z,\lambda)\in\overline{\mathbb{D}}$ for all 
$z,\lambda\in\mathbb{D}$. Therefore $\SE{F}\in\cal{S}_2$ as required.
\end{proof}

\begin{remark} {\em In Definition \ref{SEf_all}, when either $F_{21}=0$ or $F_{12}=0$, the function 
\[\SEf(F)(z,\lambda)=-\cal{F}_{F(\lambda)}(z)=-F_{11}(\lambda),\]
 is independent of $z$, and so in general the map $\SEf$ can lose some information about $F$. However, 
in the case of the symmetrised bidisc, {\em no} information is lost; see Remark \ref{SE-LeftNG}. 
}
\end{remark}
\section{Relations between $\cal{S}^{2\times2}$ and the set of analytic kernels on $\mathbb{D}^2$} \label{kernels-maps}

Basic notions and statements on analytic kernels can be found in the book \cite{AMc02} and in    Aronszajn's  paper \cite{aron}.

Let  $N$ and $M$ be analytic kernels on $\mathbb{D}^2$, and let $K_{N,M}$ be the hermitian symmetric function on $\mathbb{D}^2\times\mathbb{D}^2$ given by 
\[K_{N,M}(z,\lambda,w,\mu)=1-(1-\overline{w}z)N(z,\lambda,w,\mu)-(1-\overline{\mu}\lambda)M(z,\lambda,w,\mu)\]
for all $z,\lambda,w,\mu\in\mathbb{D}$.

We define the set $\cal{R}_1$ to be 
\begin{align}\label{R_1}
\cal{R}_1:=
\{(N,M): \; N,M,K_{N,M}\text{ are analytic kernels on }\mathbb{D}^2\text{ and }K_{N,M} \text{ is of rank $1$}\}.\end{align}


\subsection{The map $\UEf:\cal{S}^{2\times2}\to\cal{R}_{1}$}
\label{UEf}

For every $F=\FF\in\cal{S}^{2\times2}$ we  define functions
$\gamma$ and $\eta$ by equations 
\begin{equation}\label{gamma-eta-1}
\gamma(\lambda,z):=(1-F_{22}(\lambda)z)^{-1}F_{21}(\lambda)
\;\text{ and }\;\eta(\lambda,z):=\begin{bmatrix}1\\z\gamma(\lambda,z) \end{bmatrix}.
\end{equation}
The functions $N_F$ and $M_F$ on $\mathbb{D}^2\times\mathbb{D}^2$ are given
by
\[N_F(z,\lambda,w,\mu)=\overline{\gamma(\mu,w)}\gamma(\lambda,z)\text{ and }
M_F(z,\lambda,w,\mu)=\eta(\mu,w)^*\frac{I-F(\mu)^*F(\lambda)}{1-\overline{\mu}\lambda}\eta(\lambda,z)\]
for all $z,\lambda,w,\mu\in\mathbb{D}$. Note that, for $z,\lambda,w,\mu\in\mathbb{D}$, 
$1-F_{22}(\lambda)z\neq0$ and $1-F_{22}(\mu)w\neq0$, since $|F_{22}(\lambda)|\leq1$ and $|F_{22}(\mu)|\leq1$.
Hence both $N_F$ and $M_F$ are well defined. 

\begin{proposition}\label{UEfwelldef} 
Let $F\in\cal{S}^{2\times2}$ be such that $F_{21} \neq 0$. Then the maps $N_F$ and $M_F$ are analytic kernels on $\mathbb{D}^2$, $N_F$ is of rank $1$,  and $(N_F, M_F) \in \cal{R}_1$.
\end{proposition}

\begin{proof} By definition,
 \[N_F(z,\lambda,w,\mu)=\overline{\gamma(\mu,w)}\gamma(\lambda,z)\]
for $z,\lambda,w,\mu\in\mathbb{D}$, where $\gamma:\mathbb{D}^2\to\mathbb{C}$ is not equal to $0$. Thus 
$N_F$ is a kernel on $\mathbb{D}^2$ of rank  $1$.

Furthermore 
\[ M_F(z,\lambda,w,\mu)=\eta(\mu,w)^* \frac{I-F(\mu)^*F(\lambda)}{1-\overline{\mu}\lambda} \eta(\lambda,z),\]
for $z,\lambda,w,\mu\in\mathbb{D}$.
Clearly both $N_F$ and $M_F$ are analytic.

To prove that $(N_F, M_F) \in \cal{R}_1$ one has to check that  $K_{N,M}$ is an analytic kernel on $\mathbb{D}^2$ of rank $1$. Clearly $K_{N,M}$ is analytic. By Proposition \ref{1-calfcalf},
\begin{align*}1-\overline{\cal{F}_{F(\mu)}(w)}\cal{F}_{F(\lambda)}(z)=&
\overline{\gamma(\mu,w)}(1-\overline{w}z)\gamma(\lambda,z)+\eta(\mu,w)^*(I-F(\mu)^*F(\lambda))\eta(\lambda,z)\\
=&(1-\overline{w}z)N_F(z,\lambda,w,\mu)
+(1-\overline{\mu}\lambda)M_F(z,\lambda,w,\mu)
\end{align*}
for all $z,\lambda,w,\mu\in\mathbb{D}$.
 Therefore
\[K_{N_F,M_F}(z,\lambda,w,\mu)=
\overline{\cal{F}_{F(\mu)}(w)}\cal{F}_{F(\lambda)}(z)\] for all $z,\lambda,w,\mu\in\mathbb{D}$.
Thus $K_{N_F,M_F}$ is an analytic kernel on $\mathbb{D}^2$ of rank $1$. 
Therefore $(N_F,M_F)\in\cal{R}_{1}$. 
\end{proof}

\begin{proposition}\label{UEfwelldef_0} 
Let $F\in\cal{S}^{2\times2}$ be such that $F_{21} = 0$. Then the maps $N_F$ and $M_F$ are analytic kernels on $\mathbb{D}^2$, $N_F$ is of rank $0$,  and $(N_F, M_F) \in \cal{R}_1$.
Moreover, 
\[N_F(z,\lambda,w,\mu) =0, \quad M_F(z,\lambda,w,\mu)=
\frac{1-\overline{F_{11}(\mu)}F_{11}(\lambda)}{1-\overline{\mu}\lambda},
\]
and
\[K_{N_F,M_F}(z,\lambda,w,\mu) =\overline{F_{11}(\mu)}F_{11}(\lambda),\]
for all $z,\lambda,w,\mu\in\mathbb{D}$.
\end{proposition}

\begin{proof} For every 
$F=\begin{bmatrix}F_{11}&F_{12}\\0 &F_{22}\end{bmatrix}\in\cal{S}^{2\times2}$, 
the functions
$\gamma$ and $\eta$ are given by 
\[\gamma(\lambda,z)=(1-F_{22}(\lambda)z)^{-1}F_{21}(\lambda)=0
\; \text{ and }\;
 \eta(\lambda,z)=\begin{bmatrix}1\\z\gamma(\lambda,z) \end{bmatrix}= \begin{bmatrix}1\\0 \end{bmatrix},
\]
for all $\lambda, z\in\mathbb{D}$. Thus,
 \[N_F(z,\lambda,w,\mu)=0,\]
for $z,\lambda,w,\mu\in\mathbb{D}$, and so has rank $0$.
Furthermore 
\[M_F(z,\lambda,w,\mu)=\begin{bmatrix}1 &0 \end{bmatrix}\frac{I-F(\mu)^* F(\lambda)}{1-\overline{\mu}\lambda}
\begin{bmatrix}1 \\0 \end{bmatrix}= \frac{1-\overline{F_{11}(\mu)}F_{11}(\lambda)}{1-\overline{\mu}\lambda},\]
for $z,\lambda,w,\mu\in\mathbb{D}$, which is independent of $z$ and $w$. Hence $M_F$ is a kernel on $\mathbb{D}^2$.
Clearly both $N_F$ and $M_F$ are analytic.

It is easy to see that
\[K_{N,M}(z,\lambda,w,\mu)=1-(1-\overline{\mu}\lambda)M(z,\lambda,w,\mu)= \overline{F_{11}(\mu)}F_{11}(\lambda),\]
for all $z,\lambda,w,\mu\in\mathbb{D}$, which is independent of $z$ and $w$.
Thus $K_{N_F,M_F}$ is an analytic  kernel on $\mathbb{D}^2$  of rank $1$. 
Therefore $(N_F,M_F)\in\cal{R}_{1}$. 
\end{proof}

\begin{definition} The map
$\UEf:\cal{S}^{2\times2}\to\cal{R}_{1}$ is given by \[\UE{F}=(N_F,M_F)\] for each 
$F\in\cal{S}^{2\times2}$.
\end{definition}

By Propositions \ref{UEfwelldef} and \ref{UEfwelldef_0}, the map $\UEf$ is well defined.


\subsection{Procedure $UW$ and the set-valued  map $\UWf:\cal{R}_{11}\to\cal{S}^{2\times2}$}\label{UWf}

Let $F\in\cal{S}^{2\times2}$ be such that $F_{21}\neq 0$. Then  the kernel $N_F$ has rank $1$. In this case
$\UEf$ maps into a subset 
$\cal{R}_{11}$ of $\cal{R}_1$ rather than onto all of $\cal{R}_1$.

\begin{definition}
The subset $\cal{R}_{11}$ of $\cal{R}_1$ is given by
\begin{align*}\cal{R}_{11}:=
\{(N,M): \; N,M,K_{N,M}\text{ are analytic kernels on }\mathbb{D}^2\text{ and }N,K_{N,M}\; \text{ are of rank $1$}\}.\end{align*}
\end{definition}

By the Moore-Aronszajn Theorem \cite[Theorem 2.23]{AMc02}, for 
 each kernel $k$ on a set $X$,  there exists a unique Hilbert function space $\cal{H}_k$ on $X$ that has $k$ as its kernel.

Let us describe the procedure for the  construction of a function in 
$\cal{S}^{2\times2}$ from a pair of kernels in $\cal{R}_{11}$.

\begin{theorem}[Procedure $UW$]\label{procUWcons}
Let $(N,M)\in\cal{R}_{11}$. 
Then there are functions $f\in\cal{H}_N$ and $g\in\cal{H}_{K_{N,M}}$ such that
\[N(z,\lambda,w,\mu)=\overline{f(w,\mu)}f(z,\lambda)\text{ and }K_{N,M}(z,\lambda,w,\mu)=\overline{g(w,\mu)}g(z,\lambda)\] for all
$z,\lambda,w,\mu\in\mathbb{D}$ and
 a function $\Xi\in\cal{S}^{2\times2}$ such that
\[\Xi(\lambda)\begin{pmatrix}1\\zf(z,\lambda)\end{pmatrix}=\begin{pmatrix}g(z,\lambda)\\f(z,\lambda)\end{pmatrix}\]
for all $z,\lambda\in\mathbb{D}$.
\end{theorem}

\begin{proof}
Let $(N,M)\in \cal{R}_{11}$, so that $N,K_{N,M}$ are analytic kernels on $\mathbb{D}^2$  of rank $1$. Thus 
there are  functions $f\in \cal{H}_N$, $v_{z,\lambda}\in \cal{H}_M$ and  $g\in\cal{H}_{K_{N,M}}$ such that 
\[N(z,\lambda,w,\mu)=\overline{f(w,\mu)}f(z,\lambda), \;K_{N,M}(z,\lambda,w,\mu)=\overline{g(w,\mu)}g(z,\lambda)\] 
and
\[M(z,\lambda,w,\mu)=\langle v_{z,\lambda}, v_{w,\mu}\rangle_{\cal{H}_M}\] 
for all $z,\lambda,w,\mu\in\mathbb{D}$. 

Hence $(N,M)\in\cal{R}_{11}$ can be presented in the following form
\begin{equation}\label{KNM}
\overline{g(w,\mu)}g(z,\lambda)=
1-(1-\overline{w}z)\overline{f(w,\mu)}f(z,\lambda)-(1-\overline{\mu}\lambda)\langle v_{z,\lambda}, v_{w,\mu}\rangle_{\cal{H}_M},
\end{equation}
and so
\begin{eqnarray}\label{kernNM1}
&~&\overline{g(w,\mu)}g(z,\lambda)+\overline{f(w,\mu)}f(z,\lambda)+
\langle v_{z,\lambda},v_{w,\mu}\rangle_{\cal{H}_M} \nonumber\\
&=& 1+\overline{w}z\overline{f(w,\mu)}f(z,\lambda)
 +\overline{\mu}\lambda\langle v_{z,\lambda}, v_{w,\mu}\rangle_{\cal{H}_M}
\end{eqnarray} 
for all $z,\lambda,w,\mu\in\mathbb{D}$.
The left hand side of \eqref{kernNM1} can be written as
\begin{align*}\overline{g(w,\mu)}g(z,\lambda)+&\overline{f(w,\mu)}f(z,\lambda)+
\langle v_{z,\lambda},v_{w,\mu}\rangle_{\cal{H}_M}\\
=&\left\langle\begin{pmatrix} g(z,\lambda)\\ f(z,\lambda)\\ v_{z,\lambda}\end{pmatrix},
\begin{pmatrix}g(w,\mu)\\ f(w,\mu)\\ v_{w,\mu}\end{pmatrix}\right\rangle_{\mathbb{C}^2\oplus\cal{H}_M},\end{align*}
and the right hand side of \eqref{kernNM1} has the form
\begin{align*}1+\overline{w}z&\overline{f(w,\mu)}f(z,\lambda)+
\overline{\mu}\lambda\langle v_{z,\lambda}, v_{w,\mu}\rangle_{\cal{H}_M}\\
=&\left\langle\begin{pmatrix} 1\\ zf(z,\lambda)\\ \lambda v_{z,\lambda}\end{pmatrix},
\begin{pmatrix} 1\\ wf(w,\mu) \\ \mu v_{w,\mu}\end{pmatrix}\right\rangle_{\mathbb{C}^2\oplus\cal{H}_M}\end{align*} 
for all $\lambda,\mu,z,w\in\mathbb{D}$.
Therefore
 \[\left\langle\begin{pmatrix} g(z,\lambda)\\ f(z,\lambda)\\ v_{z,\lambda}\end{pmatrix},
\begin{pmatrix}g(w,\mu)\\ f(w,\mu)\\ v_{w,\mu}\end{pmatrix}\right\rangle_{\mathbb{C}^2\oplus\cal{H}_M}=
\left\langle\begin{pmatrix} 1\\ zf(z,\lambda)\\ \lambda v_{z,\lambda}\end{pmatrix},
\begin{pmatrix} 1\\ wf(w,\mu) \\ \mu v_{w,\mu}\end{pmatrix}\right\rangle_{\mathbb{C}^2\oplus\cal{H}_M}\]
for all $z,\lambda,w,\mu\in\mathbb{D}$.

Thus the relation \eqref{KNM} can be express by the statement that the Gramian of vectors 
\[
\begin{pmatrix} g(z,\lambda)\\ f(z,\lambda)\\ v_{z,\lambda}
\end{pmatrix} \in {\mathbb{C}^2\oplus\cal{H}_M}, \;\;
\lambda,\mu,z,w\in\mathbb{D},
\]
is equal to the Gramian of vectors
\[
\begin{pmatrix} 1\\ wf(w,\mu) \\ \mu v_{w,\mu}\end{pmatrix}
\in {\mathbb{C}^2\oplus\cal{H}_M},\;\;  \lambda,\mu,z,w\in\mathbb{D}.
\]
Hence  there is an isometry 
\[L_0:\spn\left\{\begin{pmatrix} 1\\ zf(z,\lambda)\\ \lambda v_{z,\lambda}\end{pmatrix}:
z,\lambda\in\mathbb{D}\right\}\to\mathbb{C}^2\oplus \cal{H}_M\] such that
\[L_0\begin{pmatrix} 1\\ zf(z,\lambda)\\ \lambda v_{z,\lambda}\end{pmatrix}=
\begin{pmatrix}g(z,\lambda)\\ f(z,\lambda)\\ v_{z,\lambda}\end{pmatrix}\] for all $z,\lambda\in\mathbb{D}$.

We extend $L_0$ to a contraction $L$ on $\mathbb{C}^2\oplus \cal{H}_M$ by defining $L$ to be $0$ on 
$(\mathbb{C}^2\oplus\cal{H}_M)\ominus\spn\left\{( 1 , zf(z,\lambda) , \lambda v_{z,\lambda}):z,\lambda\in\mathbb{D}\right\}$.
Write $L$ as a block operator matrix \[L=\begin{bmatrix} A & B \\ C& D\end{bmatrix}:
\mathbb{C}^2\oplus\cal{H}_M\to \mathbb{C}^2\oplus \cal{H}_M\] 
where $A:\mathbb{C}^2\to\mathbb{C}^2$, $B:\cal{H}_M\to\mathbb{C}^2$, $C:\mathbb{C}^2\to\cal{H}_M$ and $D:\cal{H}_M\to\cal{H}_M$, 
then $L$ satisfies
\[\begin{bmatrix} A & B \\ C& D\end{bmatrix}
\begin{pmatrix} \begin{pmatrix}1\\ zf(z,\lambda)\end{pmatrix}\\ \lambda v_{z,\lambda}\end{pmatrix}=
\begin{pmatrix}\begin{pmatrix}g(z,\lambda)\\ f(z,\lambda)\end{pmatrix}\\ v_{z,\lambda}\end{pmatrix}\]
for all $z,\lambda\in\mathbb{D}$.

Then, for $z,\lambda\in\mathbb{D}$, we obtain the pair of equations
\[A\begin{pmatrix}1\\zf(z,\lambda)\end{pmatrix}+B\lambda v_{z,\lambda}=\begin{pmatrix}g(z,\lambda)\\f(z,\lambda)\end{pmatrix}\]
and
\[C\begin{pmatrix}1\\zf(z,\lambda)\end{pmatrix}+D\lambda v_{z,\lambda}=v_{z,\lambda}.\]
 Since $L$ is a contraction, 
$\| D\|\le 1$ and
$I_{\cal{H}_M}-D\lambda$ is invertible for all $\lambda\in\mathbb{D}$. From the second of these equations,
\[v_{z,\lambda}=(I_{\cal{H}_M}-D\lambda )^{-1}C\begin{pmatrix}1\\zf(z,\lambda)\end{pmatrix}\]
for all $z,\lambda\in\mathbb{D}$.
Hence the first equation has the form
\[(A+B\lambda(I_{\cal{H}_M}-D\lambda )^{-1}C)\begin{pmatrix}1\\zf(z,\lambda)\end{pmatrix}
=\begin{pmatrix}g(z,\lambda)\\f(z,\lambda)\end{pmatrix}\]
for all $z,\lambda\in\mathbb{D}$.

Recall that, for the operator $L$, the linear fractional transformation 
\[\cal{F}_L(\lambda)=A+B\lambda(I_{\cal{H}_M}-D\lambda)^{-1}C\] for all $\lambda\in\mathbb{D}$.
Since $L$ is a contraction, by Proposition \ref{calFleq1} and Remark \ref{calFholo}, 
\[\|\cal{F}_L(\lambda)\|\leq1\text{ for all }\lambda\in\mathbb{D},\] and $\cal{F}_L$ is analytic on $\mathbb{D}$. Since
$A$ and $B\lambda(I_{\cal{H}_M}-D\lambda)^{-1}C$
are operators from $\mathbb{C}^2$ to $\mathbb{C}^2$, $\cal{F}_L$ is in $\cal{S}^{2\times2}$. Then  $\Xi=\cal{F}_L$ has required properties.
\end{proof}

The function $\Xi$ constructed with Procedure $UW$ is not necessarily unique since the functions $f$, $g$ and $v_{z,\lambda}$  are not uniquely defined. The following proposition gives relations between different $\Xi$ obtained using Procedure $UW$.

\begin{proposition}\label{relationbetweenfunctfromUW}
Let $(N,M)\in\cal{R}_{11}$ and let $f_1,f_2\in\cal{H}_N$,
$v^1_{z,\lambda},v^2_{z,\lambda} \in \cal{H}_M$ and 
 and $g_1,g_2\in\cal{H}_{K_{N,M}}$ be such that
\[N(z,\lambda,w,\mu)=\overline{f_1(w,\mu)}f_1(z,\lambda)=\overline{f_2(w,\mu)}f_2(z,\lambda),\]
\[M(z,\lambda,w,\mu)=\langle v^1_{z,\lambda}, v^1_{w,\mu}\rangle_{\cal{H}_M}= \langle v^2_{z,\lambda}, v^2_{w,\mu}\rangle_{\cal{H}_M},\] 
and
\[K_{N,M}(z,\lambda,w,\mu)=\overline{g_1(w,\mu)}g_1(z,\lambda)=\overline{g_2(w,\mu)}g_2(z,\lambda)\]  
for all $z,\lambda,w,\mu\in\mathbb{D}$.
Let $\Xi_1$ and $\Xi_2$ be constructed from $(N,M)$ using Procedure $UW$ with the functions $f_1,g_1, v^1$ and $f_2,g_2, v^2$,
respectively.
Then
\[\Xi_2=\begin{bmatrix}\zeta_1&0\\0&\zeta_2\end{bmatrix}\Xi_1\begin{bmatrix}1&0\\0&\overline{\zeta_2}\end{bmatrix}\]
for some $\zeta_1,\zeta_2\in\mathbb{T}$.
\end{proposition}

\begin{proof}
It is easy to see that $f_2=\zeta_f f_1$ and $g_2=\zeta_g g_1$ for some $\zeta_f,\zeta_g\in\mathbb{T}$.
By Theorem \ref{procUWcons}, $\Xi_1$ and $\Xi_2$ satisfy
\[\Xi_1(\lambda)\begin{pmatrix}1\\ zf_1(z,\lambda)\end{pmatrix}=\begin{pmatrix}g_1(z,\lambda)\\f_1(z,\lambda)\end{pmatrix}\; \text{ and }\;
\Xi_2(\lambda)\begin{pmatrix}1\\ zf_2(z,\lambda)\end{pmatrix}=\begin{pmatrix}g_2(z,\lambda)\\f_2(z,\lambda)\end{pmatrix}\]
for all $z,\lambda\in\mathbb{D}$.
Hence 
\[\Xi_2(\lambda)\begin{pmatrix}1\\ zf_2(z,\lambda)\end{pmatrix}=\Xi_2(\lambda)\begin{bmatrix}1&0\\0&\zeta_f\end{bmatrix}\begin{pmatrix}1\\ zf_1(z,\lambda)\end{pmatrix}\]
and 
\[\begin{pmatrix}g_2(z,\lambda)\\f_2(z,\lambda)\end{pmatrix}=
\begin{bmatrix}\zeta_g&0\\0&\zeta_f\end{bmatrix}\begin{pmatrix}g_1(z,\lambda)\\f_1(z,\lambda)\end{pmatrix}=
\begin{bmatrix}\zeta_g&0\\0&\zeta_f\end{bmatrix}\Xi_1(\lambda)\begin{pmatrix}1\\zf_1(z,\lambda)\end{pmatrix}\]
for all $z,\lambda\in\mathbb{D}$. Thus
\[\left(\Xi_2(\lambda)\begin{bmatrix}1&0\\0&\zeta_f\end{bmatrix}-\begin{bmatrix}\zeta_g&0\\0&\zeta_f\end{bmatrix}\Xi_1(\lambda)
\right)\begin{pmatrix}1\\zf_1(z,\lambda)\end{pmatrix}=0\]
for all $z,\lambda\in\mathbb{D}$.

Since $f_1$ is a nonzero analytic function of 2 variables,
the set of zeros of $f_1$ is nowhere dense in $\D^2$. Therefore
\[\Xi_2(\lambda)=\begin{bmatrix}\zeta_g&0\\0&\zeta_f\end{bmatrix}\Xi_1(\lambda)
\begin{bmatrix}1&0\\0&\overline{\zeta_f}\end{bmatrix}\]
for all $\lambda\in\mathbb{D}$.
\end{proof}

Proposition \ref{relationbetweenfunctfromUW} leads us to the following result. 

\begin{proposition}\label{UWfwelldefined}
Let $(N,M)\in\cal{R}_{11}$. Let $\Xi$ be any function constructed from $(N,M)$ by Procedure $UW$. 
Then 
\[\left\{\begin{bmatrix}\zeta_1&0\\0&\zeta_2\end{bmatrix}\Xi
\begin{bmatrix}1&0\\0&\overline{\zeta_2}\end{bmatrix}:\zeta_1,\zeta_2\in\mathbb{T}\right\}\subseteq\cal{S}^{2\times2}\]
is the set of all possible functions that can be constructed from $(N,M)$ by Procedure $UW$. 
 
\end{proposition}

\begin{definition}
The map $\UWf$ is the set-valued map from $\cal{R}_{11}$ to $\cal{S}^{2\times2}$ given by
\[\UW{N,M} = \left\{ \; \Xi \in\cal{S}^{2\times2}\; \text{constructed  by Procedure} \;  UW\;
\text{for }\; (N,M)\in R_{11}
 \right\}.
\]  
\end{definition}

\begin{proposition} \label{UEUW=id}
Let $(N,M)\in\cal{R}_{11}$ and let $\;\Xi\in\UW{N,M}$. Then
\[\UE{\Xi}=(N,M).\]
\end{proposition}

\begin{proof} Let $\Xi=\begin{bmatrix}a&b\\c&d\end{bmatrix}
\in\cal{S}^{2\times2}$. 
Then $\UE{\Xi}=(N_{\Xi},M_{\Xi})$, where
\[
N_{\Xi}(z,\lambda,w,\mu)=
\overline{\frac{c(\mu)}{1-d(\mu)w}}
\frac{c(\lambda)}{1-d(\lambda)z}
\]
and 
\[
M_{\Xi}(z,\lambda,w,\mu)=
\begin{bmatrix} 1&\frac{\overline{w}\,\overline{c(\mu)}}
{1-\overline{d(\mu)}\,\overline{w}}
\end{bmatrix}
\frac{I-\Xi(\mu)^*\Xi(\lambda)}{1-\overline{\mu}\lambda}
\begin{bmatrix}1\\\frac{zc(\lambda)}{1-d(\lambda)z}\end{bmatrix},
\]
for all $z,\lambda,w,\mu\in\mathbb{D}$.

By assumption, $\Xi\in\UW{N,M}$. Thus there exist functions
 $f$ and $g$ such that
\[N(z,\lambda,w,\mu)=\overline{f(w,\mu)}f(z,\lambda), \; K_{N,M}(z,\lambda,w,\mu)=\overline{g(w,\mu)}g(z,\lambda)\]
for all $z,\lambda,w,\mu\in\mathbb{D}$,
and 
\[\Xi(\lambda)\begin{pmatrix}1\\zf(z,\lambda)\end{pmatrix}=\begin{pmatrix}g(z,\lambda)\\f(z,\lambda)\end{pmatrix}\]
for all $z,\lambda\in\mathbb{D}$.

Hence \[a(\lambda)+b(\lambda)zf(z,\lambda)=g(z,\lambda)\text{ and }c(\lambda)+d(\lambda)zf(z,\lambda)=f(z,\lambda)\]
for all $z,\lambda\in\mathbb{D}$.
Therefore, for all $z,\lambda\in\mathbb{D}$, 
 $1-d(\lambda)z\neq0$ and
\[f(z,\lambda)= (1-d(\lambda)z)^{-1}c(\lambda).\] 
Thus 
\[N_{\Xi}(z,\lambda,w,\mu)=\overline{f(w,\mu)}f(z,\lambda)=N(z,\lambda,w,\mu)\] for all $z,\lambda,w,\mu\in\mathbb{D}$.
Moreover
\[\cal{F}_{\Xi(\lambda)}(z)=a(\lambda)+b(\lambda)z(1-d(\lambda)z)^{-1}c(\lambda)=g(z,\lambda)\]
for all $z,\lambda\in\mathbb{D}$.
Therefore
\[\overline{\cal{F}_{\Xi(\mu)}(w)}\cal{F}_{\Xi(\lambda)}(z)=\overline{g(w,\mu)}g(z,\lambda)=K_{N,M}(z,\lambda,w,\mu)\]
for all $z,\lambda,w,\mu\in\mathbb{D}$.
By Proposition \ref{1-calfcalf},
\[1-\overline{\cal{F}_{\Xi(\mu)}(w)}\cal{F}_{\Xi(\lambda)}(z)=(1-\overline{w}z)N_{\Xi}(z,\lambda,w,\mu)+(1-\overline{\mu}\lambda)M_{\Xi}(
z,\lambda,w,\mu),\]
and so
\[1-K_{N,M}(z,\lambda,w,\mu)=(1-\overline{w}z)N(z,\lambda,w,\mu)+(1-\overline{\mu}\lambda)M_{\Xi}(
z,\lambda,w,\mu)\]
for all $z,\lambda,w,\mu\in\mathbb{D}$.
By assumption, 
\[
K_{N,M}(z,\lambda,w,\mu)=1-(1-\overline{w}z)N(z,\lambda,w,\mu)-(1-\overline{\mu}\lambda)M(z,\lambda,w,\mu)
\]
for all $z,\lambda,w,\mu\in\mathbb{D}$.
Hence
$M_{\Xi}(z,\lambda,w,\mu)=M(z,\lambda,w,\mu)$
for all $z,\lambda,w,\mu\in\mathbb{D}$. 
\end{proof}

\begin{proposition} \label{UWUE=}
For any $F\in\cal{S}^{2\times2}$ such that $F_{21} \neq 0$,
\[\UWf\circ\UE{F}=\left\{\begin{bmatrix}\zeta_1&0\\0&\zeta_2\end{bmatrix}F\begin{bmatrix}1&0\\0&\overline{\zeta_2}\end{bmatrix}:
\zeta_1,\zeta_2\in\mathbb{T}\right\}.\]
\end{proposition}

\begin{proof}
Let $F=\FF\in\cal{S}^{2\times2}$. Then $\UE{F}=(N_F,M_F)$ where
\[N_{F}(z,\lambda,w,\mu)=\overline{\frac{F_{21}(\mu)}{1-F_{22}(\mu)w}}
\frac{F_{21}(\lambda)}{1-F_{22}(\lambda)z}\]
and 
\[M_{F}(z,\lambda,w,\mu)=\begin{bmatrix}1&\frac{\overline{wF_{21}(\mu)}}{1-\overline{F_{22}(\mu)w}}\end{bmatrix}
\frac{I-F(\mu)^*F(\lambda)}{1-\overline{\mu}\lambda}
\begin{bmatrix}1\\\frac{zF_{21}(\lambda)}{1-F_{22}(\lambda)z}\end{bmatrix},\]
for all $z,\lambda,w,\mu\in\mathbb{D}$.
By Proposition \ref{1-calfcalf},
\[1-\overline{\cal{F}_{F(\mu)}(w)}\cal{F}_{F(\lambda)}(z)=
(1-\overline{w}z)N_{F}(z,\lambda,w,\mu)+(1-\overline{\mu}\lambda)M_{F}(z,\lambda,w,\mu),\]
and so 
\[K_{N_F,M_F}(z,\lambda,w,\mu)=1-(1-\overline{w}z)N_{F}(z,\lambda,w,\mu)-(1-\overline{\mu}\lambda)M_{F}(z,\lambda,w,\mu)=
\overline{\cal{F}_{F(\mu)}(w)}\cal{F}_{F(\lambda)}(z)\]
for all $z,\lambda,w,\mu\in\mathbb{D}$.
Apply Procedure $UW$ to $(N_F,M_F)$ to construct a function $\Xi\in\cal{S}^{2\times2}$ such that 
\[\Xi(\lambda)\begin{pmatrix}1\\\frac{zF_{21}(\lambda)}{1-F_{22}(\lambda)z}\end{pmatrix}=
\begin{pmatrix}\cal{F}_{F(\lambda)}(z)\\\frac{F_{21}(\lambda)}{1-F_{22}(\lambda)z}\end{pmatrix}\]
for all $z,\lambda\in\mathbb{D}$.
Then, by Proposition \ref{UWfwelldefined},
\[\UW{N_F,M_F}=\left\{\begin{bmatrix}\zeta_1&0\\0&\zeta_2\end{bmatrix}\Xi\begin{bmatrix}1&0\\0&\overline{\zeta_2}\end{bmatrix}:
\zeta_1,\zeta_2\in\mathbb{T}\right\}.\]
Note
\begin{align*}
F(\lambda)\begin{pmatrix}1\\\frac{zF_{21}(\lambda)}{1-F_{22}(\lambda)z}\end{pmatrix}=&
\begin{bmatrix}F_{11}(\lambda)&F_{12}(\lambda)\\F_{21}(\lambda)&F_{22}(\lambda)\end{bmatrix}
\begin{pmatrix}1\\\frac{zF_{21}(\lambda)}{1-F_{22}(\lambda)z}\end{pmatrix}\\
=&
\begin{pmatrix}F_{11}(\lambda)+\frac{F_{12}(\lambda)zF_{21}(\lambda)}{1-F_{22}(\lambda)z}\\
F_{21}(\lambda)+\frac{F_{22}(\lambda)zF_{21}(\lambda)}{1-F_{22}(\lambda)z}\end{pmatrix}
=\begin{pmatrix}\cal{F}_{F(\lambda)}(z)\\\frac{F_{21}(\lambda)}{1-F_{22}(\lambda)z}\end{pmatrix},\end{align*}
for all $z,\lambda\in\mathbb{D}$.
Therefore
\[(\Xi(\lambda)-F(\lambda))\begin{pmatrix}1\\\frac{zF_{21}(\lambda)}{1-F_{22}(\lambda)z}\end{pmatrix}=0,\]
for all $z,\lambda\in\mathbb{D}$.
Since $F_{21}$ is a nonzero analytic function on $\D$,
 the zeros of $F_{21}$ are isolated  in $\D$.
Thus $\Xi(\lambda)=F(\lambda)$ for all $\lambda\in\mathbb{D}$.
Hence
\[\UWf\circ\UE{F}=\left\{\begin{bmatrix}\zeta_1&0\\0&\zeta_2\end{bmatrix}F\begin{bmatrix}1&0\\0&\overline{\zeta_2}\end{bmatrix}:
\zeta_1,\zeta_2\in\mathbb{T}\right\}.\]
\end{proof}


\subsection{The map $\RSf:\cal{R}_1\to\cal{S}_2$}\label{RSf}

\begin{definition}
The map  $\RSf$ is the set-valued map from $\cal{R}_1$ to $\cal{S}_2$ which is given, for each $(N,M)\in\cal{R}_1$, by
\[\RS{N,M}=\{f \in \cal{S}_2, \;\text{such that} \;K_{N,M}(z,\lambda,w,\mu)=\overline{f(w,\mu)}f(z,\lambda),\; z,\lambda,w,\mu\in\mathbb{D} \}.\] 
\end{definition}

\begin{proposition}\label{RSfwelldefined}
$\RSf$ is well defined and, for  $(N,M)\in\cal{R}_1$,
\[\RS{N,M}=\{\zeta f:\zeta\in\mathbb{T}\},\] where
$f:\mathbb{D}^2\to\mathbb{C}$ is analytic and satisfies
\[K_{N,M}(z,\lambda,w,\mu)=\overline{f(w,\mu)}f(z,\lambda)\] for all $z,\lambda,w,\mu\in\mathbb{D}$.

\end{proposition}

\begin{proof}
Let $(N,M)\in\cal{R}_1$. Then $K_{N,M}$ is an analytic kernel on $\mathbb{D}^2$ of rank $1$. Thus there exist an analytic function
$f:\mathbb{D}^2\to\mathbb{C}$ such that 
\[K_{N,M}(z,\lambda,w,\mu)=\overline{f(w,\mu)}f(z,\lambda)\] for all $z,\lambda,w,\mu\in\mathbb{D}$. In addition, if for an analytic function $g:\mathbb{D}^2\to\mathbb{C}$,
\[K_{N,M}(z,\lambda,w,\mu)=\overline{g(w,\mu)}g(z,\lambda)\] for all $z,\lambda,w,\mu\in\mathbb{D}$, then 
$g = \zeta f$ for some $\zeta\in\mathbb{T}$.

Note
\[1-K_{N,M}(z,\lambda,w,\mu)=
(1-\overline{w}z)N(z,\lambda,w,\mu)+(1-\overline{\mu}\lambda)M(z,\lambda,w,\mu)\geq0\]
for all $z,\lambda,w,\mu\in\mathbb{D}$. Thus
\[
1-\overline{f(w,\mu)} f(z,\lambda)=
1-K_{N,M}(z,\lambda,w,\mu) \ge 0 \]
for all $z,\lambda,w,\mu\in\mathbb{D}$. Hence
 $|f(z,\lambda)|\leq1$ for all $z,\lambda\in\mathbb{D}$.
Therefore  $f\in\cal{S}_2$, and so $\RSf$ is well defined.
\end{proof}

Let us consider relations between  $\RSf$ and other maps in the rich saltire. 

\begin{proposition} \label{RSUEequalSE}
Let $F\in\cal{S}^{2\times2}$. Then
\[\RSf\circ\UE{F}=\left\{\zeta\SE{F}:\zeta\in\mathbb{T}\right\}.\]
\end{proposition}

\begin{proof}
By the definition,
$\SE{F}(z,\lambda)=-\cal{F}_{F(\lambda)}(z)$ for all $z,\lambda\in\mathbb{D}$. By the definition of $\UE{F}$ and by Propositions \ref{UEfwelldef} and \ref{UEfwelldef_0}, 
$\UE{F}=(N_F,M_F)\in\cal{R}_1$, where
\[K_{N_F,M_F}(z,\lambda,w,\mu)=\overline{\cal{F}_{F(\mu)}(w)}\cal{F}_{F(\lambda)}(z)=
\overline{(-\cal{F}_{F(\mu)}(w))}(-\cal{F}_{F(\lambda)}(z))\]
for all $z,\lambda,w,\mu\in\mathbb{D}$. 
Thus
\[\RSf\circ\UE{F}=\RS{N_F,M_F}=\left\{\zeta\SE{F}:\zeta\in\mathbb{T}\right\}.\]
\end{proof}

\begin{proposition} Let $(N,M)\in\cal{R}_{11}$. Then 
\[\RS{N,M}=\{\SE{F}:F\in\UW{N,M}\}.\]
\end{proposition}

\begin{proof}
Let $(N,M)\in\cal{R}_{11}$ and let
$\Xi=\begin{bmatrix}\Xi_{11}&\Xi_{12}\\\Xi_{21}&\Xi_{22}\end{bmatrix}\in\cal{S}^{2\times2}$ 
be constructed  by Procedure UW for  $(N,M)$.
Then $\UW{N,M}=\left\{\begin{bmatrix}\zeta_1&0\\0&\zeta_2\end{bmatrix}\Xi\begin{bmatrix}1&0\\0&\overline{\zeta_2}\end{bmatrix}:
\zeta_1,\zeta_2\in\mathbb{T}\right\}$ and 
\begin{align*}\SEf\left(\begin{bmatrix}\zeta_1&0\\0&\zeta_2\end{bmatrix}\Xi\begin{bmatrix}1&0\\0&\overline{\zeta_2}\end{bmatrix}\right)
(z,\lambda)
=&
\SEf\left(\begin{bmatrix}\zeta_1\Xi_{11}&\zeta_1\overline{\zeta_2}\Xi_{12}\\\zeta_2\Xi_{21}&\Xi_{22}\end{bmatrix}\right)(z,\lambda)
\\=&
-\zeta_1\Xi_{11}(\lambda)-\frac{\zeta_1\overline{\zeta_2}\Xi_{12}(\lambda)\zeta_2\Xi_{21}(\lambda)z}{1-\Xi_{22}(\lambda)z}\\
=&\zeta_1\left(-\Xi_{11}(\lambda)-\frac{\Xi_{12}(\lambda)\Xi_{21}(\lambda)z}{1-\Xi_{22}(\lambda)z}\right)
=\zeta_1\SE{\Xi}(z,\lambda)\end{align*}
for all $z,\lambda\in\mathbb{D}$ and all $\zeta_1,\zeta_2\in\mathbb{T}$. 
Hence 
\[\left\{\SE{F}:F\in\UW{N,M}\right\}=\left\{\zeta \SE{\Xi}:\zeta \in\mathbb{T}\right\}.\] 

By Proposition \ref{RSUEequalSE} and Proposition \ref{UEUW=id},
$\UE{\Xi}=(N,M)$  and
\[\RS{N,M}=\RSf\circ\UE{\Xi}=\left\{\SE{F}:F\in\UW{N,M}\right\}.\] 
\end{proof}


\subsection{The map $\RNf:\cal{S}_2\to\cal{R}_1$}\label{RNf}

\begin{theorem}\textup{\cite[Theorem 11.13]{AMc02}} \label{nonconsaglerthm}
Let $\varphi\in\cal{S}_2$. Then there are kernels $N,M$ on 
$\;\mathbb{D}^2$
such that \[1-\overline{\varphi(\mu_1,\mu_2)}\varphi(\lambda_1,\lambda_2)=
(1-\overline{\mu_1}\lambda_1)N(\lambda_1,\lambda_2,\mu_1,\mu_2)+
(1-\overline{\mu_2}\lambda_2)M(\lambda_1,\lambda_2,\mu_1,\mu_2)\]
for all $\lambda_1,\lambda_2,\mu_1,\mu_2\in\mathbb{D}$.
\end{theorem}

\begin{remark} {\em The 
pair of kernels $(N,M)$ from Theorem \ref{nonconsaglerthm} are known as Agler kernels for $\varphi\in\cal{S}_2$. There are papers with  constructive proofs of the existence of 
Agler kernels. See for example \cite{ballsadovinn}, 
\cite{kelly} and \cite{kellyknese}.

One can see that, for the Agler kernels $(N,M)$ for $\varphi\in\cal{S}_2$, 
\[K_{N,M}(z,\lambda,w,\mu)=1-(1-\overline{w}z)N(z,\lambda,w,\mu)-(1-\overline{\mu}\lambda)M(z,\lambda,w,\mu)=
\overline{\varphi(w,\mu)}\varphi(z,\lambda)\] for all $z,\lambda,w,\mu\in\mathbb{D}$. Thus 
 $K_{N,M}$ is a kernel on $\mathbb{D}^2$ of rank $1$ and 
 $(N,M)\in\cal{R}_1$. 
Moreover, $\RS{N,M}=\{\zeta\varphi:\zeta\in\mathbb{T}\}$.
}
\end{remark}

\begin{definition}
The map $\RNf$ is the set-valued map from $\cal{S}_2$ to $\cal{R}_1$ which is given, for $\varphi\in\cal{S}_2$, by
\[\RN{\varphi}=\{(N,M) \;\text{ is a pair of Agler kernels for }\varphi\}.\]
\end{definition}

\begin{remark} {\em Let $(N,M)\in\cal{R}_1$ and let 
$f \in \cal{S}_2$ such that
\[K_{N,M}(z,\lambda,w,\mu)=\overline{f(w,\mu)}f(z,\lambda)\] for all $z,\lambda,w,\mu\in\mathbb{D}$. 
Then, for all $\varphi\in\RS{N,M}$, 
\[\RN{\varphi}=\RN{f}.\] 
Moreover $(N,M)\in\RN{f}$.
}
\end{remark}

\section{Relations between $\hol{\mathbb{D}}{\Gamma}$ 
and other objects in the rich saltire}\label{relsgammasect}

The rich saltire for the symmetrized bidisc is the following.
\begin{equation}\label{rich_str_G}
\xymatrixcolsep{7pc}
\xymatrix@R+36pt{
\quad \quad \cal{S}^{2\times 2} \quad \quad
\ar@<-3pt>[d]_{\LSfn{\G}}  
\ar[rd]^{ \SEf\quad\quad\quad\quad\quad\quad\quad} 
 \ar@<-3pt>[r]_-{\UEf}  &
  \quad \quad \cal{R}_{1}  \quad \quad 
 \ar@<-3pt>[l]_-{\UWf}  \ar@<-3pt>[d]_{\RSf} 
 \ar[ld]_{\quad\quad\quad\quad\quad\quad\quad\SWfn{\G}}\\ 
\quad \quad\hol{\mathbb{D}}{\Gamma} \quad \quad
 \ar@<-3pt>[u]_{\LNfn{\G}}  
 \ar@<-3pt>[r]_-{\LEfn{\G}}  & 
 \quad \quad \cal{S}_2 \quad \quad 
\ar@<-3pt>[l]_-{\LWfn{\G}} \ar@<-3pt>[u]_{\RNf} 
}
\end{equation}

We will define maps of the rich saltire for $\G$ and describe connections between different maps in the diagram 
\eqref{rich_str_G}.

\subsection{The maps $\LNf:\hol{\mathbb{D}}{\Gamma}\to\cal{S}^{2\times2}$ and
$\LSf:\cal{S}^{2\times2}\to\hol{\mathbb{D}}{\Gamma}$}\label{LNfandLSf}

\begin{proposition}\textup{\cite[Proposition 6.1]{ALY13}}\label{leftnconstruction}
For each $h=(s,p)\in\hol{\mathbb{D}}{\Gamma}$ there exists a unique 
$F=\FF\in\cal{S}^{2\times2}$ such that
\[h=(\tr{F},\det{F})\]
and $ F_{11}=F_{22}$, 
$|F_{12}|=|F_{21}|$ a. e. on $\mathbb{T}$, $F_{21}$ is either $0$ or outer and $F_{21}(0)\geq 0$. 
Moreover, for all $\mu,\lambda\in\mathbb{D}$ and all $w,z\in\mathbb{C}$ such that $1-F_{22}(\mu)w\neq0$ and $1-F_{22}(\lambda)z\neq0$,
\[1-\overline{\Phi(w,h(\mu))}\Phi(z,h(\lambda))=(1-\overline{w}z)\overline{\gamma(\mu,w)}\gamma(\lambda,z)+
\eta(\mu,w)^*(I-F(\mu)^*F(\lambda))\eta(\lambda,z).\]
\end{proposition}

The construction of $F$ in \cite[Proposition 6.1]{ALY13} is the following.
Let $h=(s,p)\in\hol{\mathbb{D}}{\Gamma}$ be  such that $\frac{1}{4}s^2=p$. Then
\[F=\begin{bmatrix} \frac{1}{2}s & 0 \\ 0 & \frac{1}{2}s\end{bmatrix}\]
satisfies all of the required conditions. 
Now suppose that $\frac{1}{4}s^2\neq p$. Then $\frac{1}{4}s^2- p$ is a non-zero $H^\infty$ function, and so it
has a unique inner-outer factorisation, expressible in the form $\varphi e^C=\frac{1}{4}s^2-p$, where $\varphi$ is inner, 
$e^C$ is outer and $e^C(0)\geq0$.
It follows that
\[F=\begin{bmatrix} \frac{1}{2}s & \varphi e^{\frac{1}{2}C} \\  e^{\frac{1}{2}C} & \frac{1}{2}s\end{bmatrix}\]
is the only matrix satisfying the required conditions.

\begin{definition}\label{LNf_G} The map
 $\LNf:\hol{\mathbb{D}}{\Gamma}\to\cal{S}^{2\times2}$ is given by 
$\LN{h}=F$, $h \in \hol{\mathbb{D}}{\Gamma}$, where $F$ is the unique element from $\cal{S}^{2\times2}$ 
such that
\[h=(\tr{F},\det{F})\]
and $ F_{11}=F_{22}$, 
$|F_{12}|=|F_{21}|$ a. e. on $\mathbb{T}$, $F_{21}$ is either $0$ or outer and $F_{21}(0)\geq 0$. 
\end{definition}

\begin{definition} The map
 $\LSf :\cal{S}^{2\times2}\to\hol{\mathbb{D}}{\Gamma}$ is given by 
\[F\mapsto (\tr{F},\det{F})\]
for all $F\in\cal{S}^{2\times2}$.
\end{definition}

The following is trivial.

\begin{lemma}
$\LSf\circ\LNf=\id_{\hol{\mathbb{D}}{\Gamma}}$. 
\end{lemma}

\begin{example} \label{LNLSneqid} 
$\LNf\circ\LSf\neq \id_{\cal{S}^{2\times2}}$.~~~
{\rm Consider the function $F$ on $\mathbb{D}$ defined by
\[F(\lambda)=\begin{bmatrix}\lambda^2 & 0\\0&\lambda\end{bmatrix}\] for all $\lambda\in\mathbb{D}$.
Then $F\in\cal{S}^{2\times2}$ and, for all $\lambda\in\mathbb{D}$,
\[\LS{F}(\lambda)=(\tr{F}(\lambda),\det{F}(\lambda))=(\lambda^2+\lambda,\lambda^3).\]
It is clear that $\LNf\circ\LS{F}\neq F$.
}
\end{example}

\subsection{The map $\LEf:\hol{\mathbb{D}}{\Gamma}\to\cal{S}_2$}\label{LEf-G}

\begin{definition}\label{LEf_G}
The map $\LEf:\hol{\mathbb{D}}{\Gamma}\to\cal{S}_2$ is given by 
 \[\LE{h}(z,\lambda):=\Phi(z,h(\lambda)),\; \;z,\lambda\in\mathbb{D},  \]
for  $h \in \hol{\mathbb{D}}{\Gamma}$.
\end{definition}

\begin{proposition}\label{LEhwelldefined}
The map $\LEf$ is well defined.
\end{proposition}

\begin{proof}
Let $h=(s,p)\in\hol{\mathbb{D}}{\Gamma}$. For  $(z,\lambda)\in\mathbb{D}^2$,
\[\LE{h}(z,\lambda)=\Phi(z,s(\lambda),p(\lambda))\text{ where } (s(\lambda),p(\lambda))\in\Gamma.\]
By \cite[Proposition 3.2]{ALY11}, $|s(\lambda)|\leq 2$ and, for all $w$ in a dense subset of $\mathbb{T}$, 
\[|\Phi(w,s(\lambda),p(\lambda))|\leq1.\]
Therefore \[|zs(\lambda)|<2\text{ and }|\Phi(z,s(\lambda),p(\lambda))|\leq1.\]
Hence
$2-zs(\lambda)\neq0$ and $\LE{h}(z,\lambda)\in\overline{\mathbb{D}}$. 
Since 
$h$ is analytic and maps into $\Gamma$, the map $\Phi(z,h(\lambda)),z,\lambda\in\mathbb{D}$  is analytic on $\mathbb{D}\times\Gamma$. 
Thus $\LE{h}\in\cal{S}_2$.
\end{proof}

One can ask  the question:
\beq \label{Qu}
\text{ {\em which} subset of $\schur$ corresponds to $\hol{\mathbb{D}}{\Gamma}$?}
\eeq
If $h=(s,p) \in \hol{\mathbb{D}}{\Gamma}$ then, for any fixed $\la \in \D$, the map 
\beq 
z \mapsto \Phi(z, h(\la))= \frac{2zp(\la)-s(\la)}{2-zs(\la)}= \frac{2p(\la)z-s(\la)}{-zs(\la)+2}
\eeq
is a linear fractional self-map $f(z)= \frac{a z +b}{c z +d}$ of $\D$ with the property ``$b =c$". To make the last phrase precise, say that a linear fractional map $f$ of the complex plane has the property ``$b =c$" if $f(0)\neq \infty$ and either $f$ is a constant map or, for some $a, b$ and $d$ in $\C$,
\[
f(z)= \frac{a z +b}{b z +d}\; \; \text{for all} \;\; z \in \C \cup \{ \infty\}.
\]
We shall denote the class of such functions $f$ in  $\cal{S}_2$ by $\cal{S}_2^{b=c}$.

Here is an answer to Question \eqref{Qu}.

\begin{proposition} \label{Phi-lin-frac} {\rm \cite[Proposition 5.2]{ALY13}}
Let $G$ be an analytic function on $\D^2$. There exists a function 
$h \in \hol{\mathbb{D}}{\Gamma}$ such that 
\beq \label{GPhi}
G(z,\la) =\Phi(z, h(\la))\; \; \text{for all} \;\; z, \la \in \D 
\eeq
if and only if $G \in \schur$ and, for every $\la \in \D$,
$G(\cdot, \la)$ is a linear fractional transformation with the property ``$b =c$". Moreover, if $\varphi\in\cal{S}_2^{b=c}$ then its corresponding function $h$ is unique.
\end{proposition}
\begin{proof} The first part of the statement was proved in \cite[Proposition 5.2]{ALY13}. We show here that, for every  $\varphi\in\cal{S}_2^{b=c}$, its corresponding function $h$ is unique.
Suppose  $g\in\hol{\mathbb{D}}{\Gamma}$ also satisfies the required properties. Then
\[\Phi(z,h(\lambda))=\varphi(z,\lambda)=\Phi(z,g(\lambda))\text{ for all }z,\lambda\in\mathbb{D}.\]
Suppose $h=(s,p)$ and $g=(q,r)$, then,  for all $z,\lambda\in\mathbb{D}$,
\[(2zp(\lambda)-s(\lambda))(2-zq(\lambda))=(2zr(\lambda)-q(\lambda))(2-zs(\lambda)).\]
Thus,  for all $z,\lambda\in\mathbb{D}$,
\[ z^2(r(\lambda)s(\lambda)-p(\lambda)q(\lambda))-2z(r(\lambda)-p(\lambda))+(q(\lambda)-s(\lambda))=0.\]
Hence, for all $\lambda\in\mathbb{D}$,
$q(\lambda)-s(\lambda)=0$ and $r(\lambda)-p(\lambda)=0$, and so
 $h=g$.
\end{proof}


\subsection{The map $\LWf:\cal{S}_2^{b=c}\to\hol{\mathbb{D}}{\Gamma}$}\label{LWf-G}

We are interested in a map from $\cal{S}_2^{b=c}$ rather than from the whole of $\cal{S}_2$. The proof of Proposition \ref{Phi-lin-frac} provides for each $\varphi\in\cal{S}_2^{b=c}$ the construction of
a unique $h_{\varphi}\in\hol{\mathbb{D}}{\Gamma}$. 

\begin{definition} For every $\varphi\in\cal{S}_2^{b=c}$ such that 
$\varphi(z,\lambda)=\frac{a(\lambda)z+b(\lambda)}{b(\lambda)z+d(\lambda)}$,
$z,\lambda\in\mathbb{D}$, with $d(\lambda)\neq0$ we define 
\[
h_\varphi(\lambda)=\left(-2\frac{b(\lambda)}{d(\lambda)},\frac{a(\lambda)}{d(\lambda)}\right), \; \lambda\in\mathbb{D}.\]
The map
$\LWf:\cal{S}_2^{b=c}\to\hol{\mathbb{D}}{\Gamma}$ is given by
\[\LW{\varphi}=h_{\varphi}\] for all $\varphi\in\cal{S}_2^{b=c}$.
\end{definition}

By Proposition \ref{Phi-lin-frac}, $\LWf$ is well defined.

\begin{proposition}\label{LEINVLW} The map $\LWf$ is the inverse of $\; \LEf: \hol{\mathbb{D}}{\Gamma} \to \cal{S}_2^{b=c}$.\end{proposition}

\begin{proof} 
Let $h=(s,p)\in\hol{\mathbb{D}}{\Gamma}$. 
Then $\LE{h}\in\cal{S}_2^{b=c}$ and 
\[\LE{h}(z,\lambda)=\Phi(z,h(\lambda))=\frac{2zp(\lambda)-s(\lambda)}{2-zs(\lambda)}=
\frac{p(\lambda)z-\frac{1}{2}s(\lambda)}{-\frac{1}{2}s(\lambda)z+1}\] 
for all $z,\lambda\in\mathbb{D}$. 
Hence by definition \[\LWf\circ\LE{h}=(-2(-\frac{1}{2}s),p)=h.\] 

Let $\varphi\in\cal{S}_2^{b=c}$ such that 
$\varphi(z,\lambda)=\frac{a(\lambda)z+b(\lambda)}{b(\lambda(z)+d(\lambda)}$, $z,\lambda\in\mathbb{D}$, with $d(\lambda)\neq0$. 
Then \[\LW{\varphi}=h_{\varphi}=\left(-2\frac{b}{d},\frac{a}{d}\right),\] and so
\[\LE{h_{\varphi}}(z,\lambda)=\Phi(z,h_{\varphi}(\lambda))=
\frac{\frac{a(\lambda)}{d(\lambda)}z-\frac{1}{2}(-2\frac{b(\lambda)}{d(\lambda)})}{1-\frac{1}{2}(-2\frac{b(\lambda)}{d(\lambda)})z}=
\frac{a(\lambda)z+b(\lambda)}{b(\lambda)z+d(\lambda)}
=\varphi(z,\lambda)\] for all $z,\lambda\in\mathbb{D}$.
Thus $\LEf\circ\LW{\varphi}=\varphi$ for all $\varphi\in\cal{S}_2^{b=c}$.
Therefore $\LWf$ is the inverse of $\LEf$.
\end{proof}

Let us consider how the defined maps interact with each other.

\begin{proposition}\label{SELNLE} The following holds $\SEf\circ\LNf=\LEf$.
\end{proposition}

\begin{proof}
Let $h\in\hol{\mathbb{D}}{\Gamma}$. 
Then, by Proposition \ref{leftnconstruction}, for $\LN{h}=F\in\cal{S}^{2\times2}$,
\[\SE{F}(z,\lambda)=-\cal{F}_{F(\lambda)}(z)=\Phi(z,h(\lambda))\] for all $z,\lambda\in\mathbb{D}$.
Hence $\SEf\circ\LN{h}(z,\lambda)=\Phi(z,h(\lambda))$ for all $z,\lambda\in\mathbb{D}$.
By definition, $\LE{h}(z,\lambda)=\Phi(z,h(\lambda))$ for all $z,\lambda\in\mathbb{D}$. 
Thus, for all $h\in\hol{\mathbb{D}}{\Gamma}$,  $\SEf\circ\LNf (h)=\LEf (h)$.
\end{proof}

\begin{corollary} \label{SELNLW=id} The following equalities hold    
$\;\;\SEf\circ\LNf\circ\LWf=\id_{\cal{S}_2^{b=c}}$
and \newline $\LWf\circ\SEf\circ\LNf=\id_{\hol{\mathbb{D}}{\Gamma}}$.\end{corollary}

\begin{proof} 
By Proposition \ref{SELNLE}, $\SEf\circ\LNf=\LEf$ and, by Proposition \ref{LEINVLW}, 
$\LWf$ is the inverse of $\LEf$. The results follow immediately.
\end{proof}

\begin{proposition} 
For all $F=\FF\in\cal{S}^{2\times2}$ 
such that $F_{11}=F_{22}$, we have \[\LEf\circ\LS{F}=\SE{F}.\]
\end{proposition}

\begin{proof}
Let $F=\FF\in\cal{S}^{2\times2}$.
Then \[\SE{F}(z,\lambda)=-F_{11}(\lambda)-\frac{F_{12}(\lambda)F_{21}(\lambda)z}{1-F_{11}(\lambda)z}
=\frac{-F_{11}(\lambda)+(F_{11}(\lambda)^2-F_{12}(\lambda)F_{21}(\lambda))z}{1-F_{11}(\lambda)z}\] 
for all $z,\lambda\in\mathbb{D}$
and
$\LS{F}=(\tr{F},\det{F})=(2F_{11},F_{11}^2-F_{21}F_{12})$. 
Thus
\begin{align*}\LEf\circ\LS{F}(z,\lambda)&=\Phi(z,2F_{11}(\lambda),F_{11}(\lambda)^2-F_{21}(\lambda)F_{12}(\lambda))\\&=
\frac{2z(F_{11}^2(\lambda)-F_{21}(\lambda)F_{12}(\lambda))-2F_{11}(\lambda)}{2-2zF_{11}(\lambda)}\\&=
\frac{-F_{11}(\lambda)+(F_{11}(\lambda)^2-F_{12}(\lambda)F_{21}(\lambda))z}{1-F_{11}(\lambda)z}
\end{align*} for all $z,\lambda\in\mathbb{D}$.
Therefore, for all $F\in\cal{S}^{2\times2}$ 
such that $F_{11}=F_{22}$, $\LEf\circ\LS{F}=\SE{F}$.
\end{proof}

However for an arbitrary $F\in\cal{S}^{2\times2}$ we may have $\LEf\circ\LS{F}\neq\SE{F}$ as the following example shows.

\begin{example}
Let $F=\begin{bmatrix}f&0\\0&g\end{bmatrix}$, 
where $f(z)$ is the Blaschke factor $B_{\half}$ and 
 $g(z)$ is the Blaschke factor $B_{-\half}$.
Then $F\in\cal{S}^{2\times2}$. It is easy to see that
\[\SE{F}(0,\lambda)=-F_{11}(\lambda)-\frac{F_{12}(\lambda)F_{21}(\lambda)\cdot0}{1-F_{22}(\lambda)\cdot0}=-f(\lambda)\]
and
\begin{align*}\LEf\circ\LSf\,\left(F\right)(0,\lambda)=&\frac{2\cdot0\cdot\det{F(\lambda)}-\tr{F(\lambda)}}{2-0\cdot\tr{F(\lambda)}}\\
=&\frac{-(f(\lambda)+g(\lambda))}{2}\end{align*}
for all $\lambda\in\mathbb{D}$. Therefore
 $\LEf\circ\LS{F}\neq\SE{F}$.
\end{example}

\begin{remark}\label{SE-LeftNG} {\em In Definition \ref{SEf_all}, when either $F_{21}=0$ or $F_{12}=0$, the function 
\[\SEf(F)(z,\lambda)=-\cal{F}_{F(\lambda)}(z)=-F_{11}(\lambda),\]
 is independent of $z$, and so in general the map $\SEf$ can lose some information about $F$. However, 
in the case of the symmetrised bidisc, {\em no} information is lost.  For $h=(s,p) \in \hol{\mathbb{D}}{\Gamma}$ such that $s^2= 4p$, by Definition \ref{LEf_G},
 \[
\LE{h}(z,\lambda):=\Phi(z,h(\lambda))=- \frac{s(\lambda)}{2}, \quad \mbox{ for } z,\lambda\in\mathbb{D}.  
\]
Secondly, by Definition \ref{LNf_G},
$\LN{h}=F$, where
\[
F=\begin{bmatrix} \tfrac 12 s & 0 \\ 0 & \tfrac 12 s\end{bmatrix}.
\]
 Therefore,  for $h=(s,p) \in \hol{\mathbb{D}}{\Gamma}$ such that $h(\D)\subset \mathcal{R}$,
\[
\SEf \circ \LN{h}(z,\lambda)=\LE{h}(z,\lambda)= -\frac{1}{2}s(\lambda), \; \lambda \in\mathbb{D}.
\]
}
\end{remark}


\subsection{The map $\SWf:\cal{R}_{11}\to\hol{\mathbb{D}}{\Gamma}$}\label{SWf-G}

\begin{definition}
The map  $\SWf$ is the set-valued map from $\cal{R}_{11}$ to $\hol{\mathbb{D}}{\Gamma}$ which is given by 
\[
\SWf{(N,M)} = \{\LS{F}:F\in\UW{N,M}\}.
\]
\end{definition}

\begin{proposition}\label{SWwelldefined}
Let $(N,M)\in\cal{R}_{11}$, and let $\;\Xi\;$ be a function constructed  by Procedure UW for $(N,M)$.
Then
\begin{align*}\{\LS{F}:F\in\UW{N,M}\}
=&\left\{\left(\tr{\begin{bmatrix}\zeta&0\\0&1\end{bmatrix}\Xi},\zeta\det{\Xi}\right):\zeta\in\mathbb{T}\right\}
\subseteq\hol{\mathbb{D}}{\Gamma}.\end{align*}
\end{proposition}

\begin{proof}  By Proposition \ref{UWfwelldefined},
\[\UW{N,M}=\left\{\begin{bmatrix}\zeta_1&0\\0&\zeta_2\end{bmatrix}\Xi\begin{bmatrix}1&0\\0&\overline{\zeta_2}\end{bmatrix}:
\zeta_1,\zeta_2\in\mathbb{T}\right\}.\]
Hence, for $F\in\UW{N,M}$, 
 $F=\begin{bmatrix}\zeta_1&0\\0&\zeta_2\end{bmatrix}\Xi
\begin{bmatrix}1&0\\0&\overline{\zeta_2}\end{bmatrix}$ for some $\zeta_1,\zeta_2\in\mathbb{T}$.
 Then
\begin{align*}\LS{F}=&\left(\tr{\begin{bmatrix}\zeta_1&0\\0&\zeta_2\end{bmatrix}\Xi
\begin{bmatrix}1&0\\0&\overline{\zeta_2}\end{bmatrix}},\det{\begin{bmatrix}\zeta_1&0\\0&\zeta_2\end{bmatrix}\Xi
\begin{bmatrix}1&0\\0&\overline{\zeta_2}\end{bmatrix}}\right)
=\left(\tr{\begin{bmatrix}\zeta_1&0\\0&1\end{bmatrix}\Xi},\zeta_1\det{\Xi}
\right).\end{align*}
\end{proof}

Therefore, for $(N,M)\in\cal{R}_{11}$,
\[\SW{N,M}=\left\{\left(\tr{\begin{bmatrix}\zeta&0\\0&1\end{bmatrix}\Xi},\zeta\det{\Xi}\right):\zeta\in\mathbb{T}\right\},\]
where $\Xi\in\cal{S}^{2\times2}$ is a function constructed  by Procedure UW for $(N,M)$. The later set is independent of the choice of $\Xi$.

Relations between  $\SWf$ and other maps in the rich saltire are the following.

\begin{proposition} \label{SWcircUE}
Let $F\in\cal{S}^{2\times2}$ such that $F_{21} \neq 0$. Then
\[\SWf\circ\UE{F}=\left\{\LSf\left(\begin{bmatrix}\zeta&0\\0&1\end{bmatrix}F\right):\zeta\in\mathbb{T}\right\}.\]
\end{proposition}

\begin{proof}
 By Proposition \ref{UWUE=},
\[\UWf\circ\UE{F}=\left\{\begin{bmatrix}\zeta_1&0\\0&\zeta_2\end{bmatrix}F\begin{bmatrix}1&0\\0&\overline{\zeta_2}\end{bmatrix}:
\zeta_1,\zeta_2\in\mathbb{T}\right\},\] 
and hence 
\begin{align*}\SWf\circ\UE{F}
=&\left\{\LSf\left(\begin{bmatrix}\zeta_1&0\\0&\zeta_2\end{bmatrix}F\begin{bmatrix}1&0\\0&\overline{\zeta_2}\end{bmatrix}\right):
\zeta_1,\zeta_2\in\mathbb{T}\right\}\\
=&\left\{\left(\tr{\begin{bmatrix}\zeta_1&0\\0&\zeta_2\end{bmatrix}F\begin{bmatrix}1&0\\0&\overline{\zeta_2}\end{bmatrix}},
\det{\begin{bmatrix}\zeta_1&0\\0&\zeta_2\end{bmatrix}F\begin{bmatrix}1&0\\0&\overline{\zeta_2}\end{bmatrix}}\right):
\zeta_1,\zeta_2\in\mathbb{T}\right\}\\
=&\left\{\LSf\left(\begin{bmatrix}\zeta&0\\0&1\end{bmatrix}F\right):
\zeta\in\mathbb{T}\right\}.\end{align*}
\end{proof}

\begin{corollary} \label{SWUELN}
Let $h=(s,p)\in\hol{\mathbb{D}}{\Gamma}$ such that  $\frac{1}{4}s^2\neq p$. 
Then \[\SWf\circ\UEf\circ\LN{h}=\left\{\left(\frac{1}{2}(\zeta+1)s,\zeta p\right):\zeta\in\mathbb{T}\right\}.\]
\end{corollary}

\begin{proof} By Definition \ref{LNf_G},
$\LN{h}= F=\begin{bmatrix}\frac{1}{2}s&F_{12}\\F_{21}&\frac{1}{2}s\end{bmatrix}$,
where $F_{21} \neq 0$ and $\det{F}=p$.
By Proposition \ref{SWcircUE},
\begin{align*}\SWf\circ\UE{F}
=&\left\{\LSf\left(\begin{bmatrix}\zeta&0\\0&1\end{bmatrix}F\right):\zeta\in\mathbb{T}\right\}\\
=&\left\{\LSf\left(\begin{bmatrix}\zeta \frac{1}{2}s&\zeta F_{12}\\F_{21}&\frac{1}{2}s\end{bmatrix}\right):\zeta\in\mathbb{T}\right\}\\
=&\left\{\left(\frac{1}{2}(\zeta+1) s,\zeta \det{F}\right):\zeta\in\mathbb{T}\right\}.\end{align*}
Therefore $\SWf\circ\UEf\circ\LN{h}=\left\{\left(\frac{1}{2}(\zeta+1) s,\zeta p\right):\zeta\in\mathbb{T}\right\}$.
\end{proof}

\begin{remark} {\em By Corollary \ref{SWUELN}, for  $h=(s,p) \in \hol{\mathbb{D}}{\Gamma}$ such that $h(\D)$ is not in $\mathcal{R}$, we have $h\in\SWf\circ\UEf\circ\LN{h}$, since,  for $\zeta=1$, 
$$\left(\frac{1}{2}(\zeta+1)s,\zeta p\right)=(s,p).$$
}
\end{remark}

\begin{corollary}
Let $\varphi\in\cal{S}_2^{b=c}$. Then \[\RSf\circ\UEf\circ\LNf\circ\LW{\varphi}=\left\{\zeta\varphi:\zeta\in\mathbb{T}\right\}.\]
\end{corollary}

\begin{proof} 
By Corollary \ref{SELNLW=id}, 
\[\SEf\circ\LNf\circ\LW{\varphi}=\varphi.\]
It is obvious that $\LNf\circ\LW{\varphi}\in\cal{S}^{2\times2}$. 
By Proposition \ref{RSUEequalSE}, 
\[\RSf\circ\UE{\LNf\circ\LW{\varphi}}=\left\{\zeta\SE{\LNf\circ\LW{\varphi}}:\zeta\in\mathbb{T}\right\}\]
Therefore 
$\RSf\circ\UEf\circ\LNf\circ\LW{\varphi}=\left\{\zeta\varphi:\zeta\in\mathbb{T}\right\}.$
\end{proof}

\section{Relations between $\hol{\mathbb{D}}{\overline{\mathcal{E}}}$ and other objects in the rich saltire}\label{relstetrasect}

The rich saltire for the tetrablock is the following.
\begin{equation}\label{rich_str_E}
\xymatrixcolsep{7pc}
\xymatrix@R+36pt{
\quad \quad \cal{S}^{2\times 2} \quad \quad
\ar@<-3pt>[d]_{\LSfn{\E}}  
\ar[rd]^{ \SEf\quad\quad\quad\quad\quad\quad\quad} 
 \ar@<-3pt>[r]_-{\UEf}  &
  \quad \quad \cal{R}_{1}  \quad \quad 
 \ar@<-3pt>[l]_-{\UWf}  \ar@<-3pt>[d]_{\RSf} 
 \ar[ld]_{\quad\quad\quad\quad\quad\quad\quad\SWfn{\E}}\\ 
\quad \quad\hol{\mathbb{D}}{\overline{\mathcal{E}}} \quad \quad
 \ar@<-3pt>[u]_{\LNfn{\E}}  
 \ar@<-3pt>[r]_-{\LEfn{\E}}  & 
 \quad \quad \cal{S}_2 \quad \quad 
\ar@<-3pt>[l]_-{\LWfn{\E}} \ar@<-3pt>[u]_{\RNf} 
}
\end{equation}

We will define the maps of the rich saltire which depend on $\E$ and describe connections between the different maps in diagram \eqref{rich_str_E}.


\subsection{The map $\LNfn{\E}:\hol{\mathbb{D}}{\overline{\mathcal{E}}}\to\cal{S}^{2\times2}$}\label{LNf-E-sect}

\begin{theorem}\label{leftntetrawelldefinedprop} 
Let $x=(x_1,x_2,x_3)\in\hol{\mathbb{D}}{\overline{\mathcal{E}}}$. There exists a unique function
\[F=\FF\in\cal{S}^{2\times2}\] such that 
\begin{equation} \label{x=pi(F)}
x=(F_{11},F_{22},\det F),
\end{equation}
and
\begin{equation} \label{add-prop}
|F_{12}|=|F_{21}|\; \text{a. e. on} \; \mathbb{T}, \; F_{21} \; \text{is either} \; 0 \; \text{or outer, and} \; F_{21}(0) \geq 0.
\end{equation}  
Moreover,
for all $\mu,\lambda\in\mathbb{D}$ and all $\;w,z\in\mathbb{C}$ such that 
\[1-F_{22}(\mu)w\neq0 \; \text{ and} \;1-F_{22}(\lambda)z\neq0, \]
\begin{align}\label{realisE}
1-\overline{\Psi(w,x(\mu))}\Psi(z,x(\lambda))
=&(1-\overline{w}z)\overline{\gamma(\mu,w)}\gamma(\lambda,z) \nn\\
~&+\eta(\mu,w)^*(I-F(\mu)^*F(\lambda))\eta(\lambda,z),
\end{align}
where
\begin{equation}\label{gamma-eta}
\gamma(\lambda,z):=(1-F_{22}(\lambda)z)^{-1}F_{21}(\lambda)
\;\text{ and }\;\eta(\lambda,z):=\begin{bmatrix}1\\z\gamma(\lambda,z) \end{bmatrix}.
\end{equation}
\end{theorem}

\begin{proof} Consider first the case that $x_1x_2=x_3$.
By Proposition \ref{cond_Ebar}, $|x_1(\lambda)|,|x_2(\lambda)|\leq1$ for all $\lambda\in\mathbb{D}$.
Then the function \[F=\begin{bmatrix}x_1&0\\0&x_2\end{bmatrix}\] is in $\cal{S}^{2\times2}$ and has the required properties \eqref{x=pi(F)} and \eqref{add-prop}, and moreover it is the only function with these properties.

In the case that  $x_1x_2\neq x_3$, the  $H^\infty$ function 
 $x_1x_2-x_3$ is nonzero, and so it
has a unique inner-outer factorisation, say $\varphi e^C=x_1x_2-x_3$ where $\varphi$ is inner, 
$e^C$ is outer and $e^C(0)\geq0$.
Let
\begin{equation} \label{F_E}
F\df \begin{bmatrix} x_1 & \varphi e^{\frac{1}{2}C} \\  e^{\frac{1}{2}C} & x_2\end{bmatrix}.
\end{equation}
One can see that
\[\det F=x_1x_2-\varphi e^C=x_1x_2-x_1x_2+x_3=x_3,\]
and $|F_{12}|=e^{\re{\frac{1}{2}C}}=|F_{21}|$ a. e. on $\mathbb{T}$, $F_{21}$ is outer, and $F_{21}(0) \geq 0$.
It follows that $F$
is the only matrix satisfying the  required properties \eqref{x=pi(F)} and \eqref{add-prop}.

Let us check that $F\in\cal{S}^{2\times2}$.
 Clearly $F$ is holomorphic on $\D$.
We must  show that $\|F(\lambda)\|\leq1$ for all $\lambda\in\mathbb{D}$. Let us prove that 
$I-F(\lambda)^*F(\lambda)$ is positive semidefinite for all $\lambda\in\mathbb{D}$.  
It is enough to show that, for all $\lambda\in\mathbb{D}$, the diagonal entries of
$I-F(\lambda)^*F(\lambda)$ are non-negative and $\det{(I-F(\lambda)^*F(\lambda))}\geq0$.
Since $|F_{12}|=|F_{21}|$ a. e. on $\mathbb{T}$ and $F_{21}F_{12}=x_1x_2-x_3$ we have
\[|F_{12}|^2=|F_{21}|^2=|F_{21}F_{12}|=|x_1x_2-x_3|\] a. e. on $\mathbb{T}$.
At  almost every $\lambda\in\mathbb{T}$,
\begin{align*} \label{I-F*F-E}
I-F(\lambda)^*F(\lambda)=\hspace{6cm}&~\\
\begin{bmatrix}1-|x_1(\lambda)|^2-|x_1(\lambda)x_2(\lambda)-x_3(\lambda)|& 
& -\overline{x_1(\lambda)}F_{12}(\lambda)
-\overline{F_{21}(\lambda)}x_2(\lambda)\\
 & & \\
-\overline{F_{12}(\lambda)}x_1(\lambda)-\overline{x_2(\lambda)}F_{21}(\lambda)& &
1-|x_1(\lambda)x_2(\lambda)-x_3(\lambda)|-|x_2(\lambda)|^2\end{bmatrix}
\end{align*}
and 
\[\det{(I-F(\lambda)^*F(\lambda))}=1-|x_1(\lambda)|^2-2|x_1(\lambda)x_2(\lambda)-x_3(\lambda)|-|x_2(\lambda)|^2+
|x_3(\lambda)|^2.\]
Let $D_{11}$ and $D_{22}$ be the diagonal entries of $I-F^*F$.
Since $x(\lambda)\in\overline{\mathcal{E}}$ for  $\lambda\in\mathbb{D}$, by Proposition \ref{cond_Ebar}, 
\[|x_2(\lambda)-\overline{x_1(\lambda)}x_3(\lambda)|+|x_1(\lambda)x_2(\lambda)-x_3(\lambda)|\leq1-|x_1(\lambda)|^2\] and 
\[|x_1(\lambda)-\overline{x_2(\lambda)}x_3(\lambda)|+|x_1(\lambda)x_2(\lambda)-x_3(\lambda)|\leq1-|x_2(\lambda)|^2\]
for all $\lambda\in\mathbb{D}$. 
Thus, for almost every $\lambda\in\mathbb{T}$,
\[D_{11}(\lambda)\geq|x_2(\lambda)-\overline{x_1(\lambda)}x_3(\lambda)|\geq0\text{ and }
D_{22}(\lambda)\geq|x_1(\lambda)-\overline{x_2(\lambda)}x_3(\lambda)|\geq0.\]
By Proposition \ref{cond_Ebar}, 
\[|x_1(\lambda)|^2+|x_2(\lambda)|^2-|x_3(\lambda)|^2+2|x_1(\lambda)x_2(\lambda)-x_3(\lambda)|\leq1,\]
for all $\lambda\in\mathbb{D}$. Hence,  for almost every $\lambda\in\mathbb{T}$,
\[\det{(I-F(\lambda)^*F(\lambda))}\geq0.\] 
Therefore
\[I-F(\lambda)^*F(\lambda)\]
for almost every $\lambda\in\mathbb{T}$. Thus 
$\|F(\lambda)\|\leq1$ for almost every $\lambda\in\mathbb{T}$, 
and so, by the Maximum Modulus Principle,  $\|F(\lambda)\|\leq1$ for all $\lambda\in\mathbb{D}$.

We now prove the  identity \eqref{realisE}.
By Proposition \ref{1-calfcalf}, for any $F=\FF\in\cal{S}^{2\times2}$, 
\[1-\cal{F}_{F(\mu)}(w)^*\cal{F}_{F(\lambda)}(z)=\overline{\gamma(\mu,w)}(1-\overline{w}z)\gamma(\lambda,z)+
\eta(\mu,w)^*(I-F(\mu)^*F(\lambda))\eta(\lambda,z)\]
for all $\mu,\lambda\in\mathbb{D}$ and $w,z\in\mathbb{C}$ such that
$1-F_{22}(\mu)w\neq0$ and $1-F_{22}(\lambda)z\neq0$.

First we note that
\begin{align*}
\cal{F}_{F(\lambda)}(z)=&F_{11}(\lambda)+\frac{F_{12}(\lambda)F_{21}(\lambda)z}{1-F_{22}(\lambda)z}=
x_1(\lambda)+\frac{(x_1(\lambda)x_2(\lambda)-x_3(\lambda))z}{1-x_2(\lambda)z}\\
=&\frac{x_1(\lambda)-x_3(\lambda)z}{1-x_2(\lambda)z}=\frac{x_3(\lambda)z-x_1(\lambda)}{x_2(\lambda)z-1}=\Psi(z,x(\lambda))
\end{align*}
for all $\lambda\in\mathbb{D}$ and all $z\in\mathbb{C}$ such that $1-F_{22}(\lambda)z\neq0$. The functions $\gamma$ and $\eta$ are defined by equations \eqref{gamma-eta}. Hence
\begin{align*}
1-\overline{\Psi(w,x(\mu))}\Psi(z,x(\lambda))=&1-\cal{F}_{F(\mu)}(w)^*\cal{F}_{F(\lambda)}(z)\\
=&(1-\overline{w}z)\overline{\gamma(\mu,w)}\gamma(\lambda,z)+
\eta(\mu,w)^*(I-F(\mu)^*F(\lambda))\eta(\lambda,z)
\end{align*}
for all $\mu,\lambda\in\mathbb{D}$ and all $w,z\in\mathbb{C}$ such that $1-F_{22}(\mu)w\neq0$ and $1-F_{22}(\lambda)z\neq0$.
\end{proof}

\begin{definition}\label{LNf-E}
The map $\LNfn{\E}:\hol{\mathbb{D}}{\overline{\mathcal{E}}}\to\cal{S}^{2\times2}$ is given  by
\[
\LNn{\E}{x}=F=\FF
\]
for $x=(x_1,x_2,x_3)\in\hol{\mathbb{D}}{\overline{\mathcal{E}}}$, where
$F\in\cal{S}^{2\times2}$ such that 
$x=(F_{11},F_{22},\det F)$,
$|F_{12}|=|F_{21}|$ a. e. on $\mathbb{T}$, $F_{21}$ is either outer or $0$ 
and $F_{21}(0)\geq0$. 
\end{definition}


\subsection{The map $\LSfn{\E}:\cal{S}^{2\times2}\to\hol{\mathbb{D}}{\overline{\mathcal{E}}}$}\label{LS-E}

\begin{definition}\label{LSF-E}
The map $\LSfn{\E}:\cal{S}^{2\times2}\to\hol{\mathbb{D}}{\overline{\mathcal{E}}}$ is defined by
\[F=\FF\mapsto(F_{11},F_{22},\det{F})\]
for each $F\in\cal{S}^{2\times2}$.
\end{definition}

By Proposition \ref{cond_Ebar} and Theorem \ref{conditionsforbE}, the map  $\LSfn{\E}$ is well defined.
Relations between the maps 
$\LNfn{\E}$ and $\LSfn{\E}$ are the following.

\begin{proposition}\label{LNf-LSf-E} 

{\em (i)} The equality $\LSfn{\E}\circ\LNfn{\E}=\id_{\hol{\mathbb{D}}{\overline{\mathcal{E}}}}$ holds, and

{\em (ii)}
$\LNfn{\E}\circ\LSfn{\E}\neq\id_{\cal{S}^{2\times2}}$.
\end{proposition}

\begin{proof}
(i) Let $x=(x_1,x_2,x_3)\in\hol{\mathbb{D}}{\overline{\mathcal{E}}}$.
By Definition \ref{LNf-E},
\[
\LNn{\E}{x}=F=\FF,
\]
 where
$F\in\cal{S}^{2\times2}$ such that 
$x=(F_{11},F_{22},\det F)$,
$|F_{12}|=|F_{21}|$ a. e. on $\mathbb{T}$, $F_{21}$ is either outer or $0$ 
and $F_{21}(0)\geq0$. 
Therefore $\LSfn{\E}\circ\LNfn{\E}=\id_{\hol{\mathbb{D}}{\overline{\mathcal{E}}}}$ holds.

(ii) Let us consider the following example: the function $F$ on $\mathbb{D}$ which is 
defined by 
\[F(\lambda)=\frac{\lambda}{\sqrt{2}}\begin{bmatrix}1&0\\ 1&0\end{bmatrix}, \;\lambda\in\mathbb{D}. \] 
Clearly, $F\in\cal{S}^{2\times2}$. 
Then
\[\LSn{\E}{F}(\lambda)=(\frac{\lambda}{\sqrt{2}},0,0)\in\hol{\mathbb{D}}{\overline{\mathcal{E}}},\] 
and, by Definition \ref{LNf-E},
\[\LNfn{\E}\circ\LSn{\E}{F}(\lambda) =
\begin{bmatrix}\frac{\lambda}{\sqrt{2}}&0\\ 0&0\end{bmatrix}, \;\lambda\in\mathbb{D}.
\] 
Hence 
$\LNfn{\E}\circ\LSfn{\E}\neq\id_{\cal{S}^{2\times2}}$.
\end{proof}


\subsection{The maps $\LEfn{\E}:\hol{\mathbb{D}}{\overline{\mathcal{E}}}\to \cal{S}_2^{\mathrm{lf}}$
and $\LWfn{\E}:\cal{S}_2^{\mathrm{lf}} \to\hol{\mathbb{D}}{\overline{\mathcal{E}}}$}\label{LEf-LWf-E}

\begin{lemma}\label{pres2lflem}
Let $\varphi\in\cal{S}_2$ be such that $\varphi(\cdot,\lambda)$ is a linear fractional map for all $\lambda\in\mathbb{D}$. 
Then $\varphi$ can be written as
\[\varphi(z,\lambda)=\frac{a(\lambda)z+b(\lambda)}{c(\lambda)z+1}\]
for all $z,\lambda\in\mathbb{D}$, where $a,b,c$ are functions from $\mathbb{D}$ to $\mathbb{C}$, and $b$ is analytic on $\mathbb{D}$.
Moreover, if $c$ is analytic on $\mathbb{D}$, then so is $a$.
\end{lemma}

\begin{proof}
Let $\varphi\in\cal{S}_2$ be such that $\varphi(\cdot,\lambda)$ is a linear fractional map for all 
$\lambda\in\mathbb{D}$. Then we can write
\[\varphi(z,\lambda)=\frac{a(\lambda)z+b(\lambda)}{c(\lambda)z+d(\lambda)}\]
for all $z,\lambda\in\mathbb{D}$, where $a,b,c,d$ are functions from $\mathbb{D}$ to $\mathbb{C}$.
Since $\varphi\in\cal{S}_2$, up to cancellation, $\varphi(\cdot,\lambda)$ does not have a pole at $0$ for any $\lambda\in\mathbb{D}$.
Thus, without loss of generality, we may write \[\varphi(z,\lambda)=\frac{a(\lambda)z+b(\lambda)}{c(\lambda)z+1}\]
for all $z,\lambda\in\mathbb{D}$. Moreover, since $b(\lambda)=\varphi(0,\lambda)$ for all $\lambda\in\mathbb{D}$, and so $b$ is analytic on $\mathbb{D}$. 

Suppose $c$ is analytic on $\mathbb{D}$.
Then \[a(\lambda)z=\varphi(z,\lambda)(c(\lambda)z+1)-b(\lambda)\] for all $z,\lambda\in\mathbb{D}$, 
and so $a$ is analytic on $\mathbb{D}$.
\end{proof}

\begin{definition}
Let $\cal{S}_2^{\mathrm{lf}}$\index{$\cal{S}_2^{\mathrm{lf}}$}  be the subset of $\cal{S}_2$ which contains
those $\varphi$ for which $\varphi(\cdot,\lambda)$ is a linear fractional map of the form
\[\varphi(z,\lambda)=\frac{a(\lambda)z+b(\lambda)}{c(\lambda)z+1}\] for all $z,\lambda\in\mathbb{D}$, where
$c$ is analytic on $\mathbb{D}$, and if $\;a(\lambda)=b(\lambda)c(\lambda)$ for some $\lambda\in\mathbb{D}$, then, in addition,
$|c(\lambda)|\leq1$.
\end{definition}

\begin{proposition}\label{varphiupsilon}
Let $\varphi$ be a function on $\mathbb{D}^2$. Then $\varphi\in\cal{S}_2^{\mathrm{lf}}$ if and only if there exists a function
$x\in\hol{\mathbb{D}}{\overline{\mathcal{E}}}$ such that \[\varphi(z,\lambda)=\Psi(z,x(\lambda))\text{ for all }z,\lambda\in\mathbb{D}.\]
\end{proposition}

\begin{proof} 
Suppose  $\varphi\in\cal{S}_2^{\mathrm{lf}}$. Then  
\[\varphi(z,\lambda)=\frac{a(\lambda)z+b(\lambda)}{c(\lambda)z+1}\] 
for all $z,\lambda\in\mathbb{D}$,
where
$c$ is analytic on $\mathbb{D}$, and if $a(\lambda)=b(\lambda)c(\lambda)$ for some $\lambda\in\mathbb{D}$, then in addition
$|c(\lambda)|\leq1$. By Lemma \ref{pres2lflem}, both $a$ and $b$ are also  analytic on $\mathbb{D}$.

Set \[x(\lambda)=(b(\lambda),-c(\lambda),-a(\lambda))\] for all $\lambda\in\mathbb{D}$.
Then $x$ is analytic on $\mathbb{D}$, and
$|\Psi(z,x(\lambda)|=|\frac{x_3(\lambda)z-x_1(\lambda)}{x_2(\lambda)z-1}|=|\varphi(z,\lambda)|\leq1$ for all $z,\lambda\in\mathbb{D}$, and if $a(\lambda)=b(\lambda)c(\lambda)$ for some $\lambda\in\mathbb{D}$, 
then, in addition, $|c(\lambda)|\leq1$. Hence, by Proposition \ref{cond_Ebar}(3), $x(\lambda)\in\overline{\mathcal{E}}$
for all $\lambda\in\mathbb{D}$, and 
\[\varphi(z,\lambda)=\Psi(z,x(\lambda))\text{ for all }z,\lambda\in\mathbb{D}.\]

Conversely, suppose there exists an 
$x=(x_1,x_2,x_3)\in\hol{\mathbb{D}}{\overline{\mathcal{E}}}$ such that 
$\varphi(z,\lambda)=\Psi(z,x(\lambda))$ for all $z,\lambda\in\mathbb{D}$. 
Then \[\varphi(z,\lambda)=\frac{x_3(\lambda)z-x_1(\lambda)}{x_2(\lambda)z-1}\] for all $z,\lambda\in\mathbb{D}$
and clearly $\varphi(\cdot,\lambda)$ is a linear fractional transformation for all $\lambda\in\mathbb{D}$. It is obvious that $x_1$, $x_2$ and $x_3$ are analytic on $\mathbb{D}$.
Since $x(\lambda)\in\overline{\mathcal{E}}$ for all $\lambda\in\mathbb{D}$,  by Proposition \ref{cond_Ebar}(3), 
$|\varphi(z,\lambda)|=|\Psi(z,x(\lambda))|\leq1$ for all $z,\lambda\in\mathbb{D}$, and 
if $x_1(\lambda)x_2(\lambda)=x_3(\lambda)$ then in addition $|x_2(\lambda)|\leq1$. Thus $\varphi\in\cal{S}_2^{\mathrm{lf}}$.
\end{proof}

By Proposition \ref{varphiupsilon}, the map below $\LEfn{\E}$ is well defined.

\begin{definition}
The map $\LEfn{\E}:\hol{\mathbb{D}}{\overline{\mathcal{E}}}\to\cal{S}_2^{\mathrm{lf}}$, for $x=(x_1,x_2,x_3)\in\hol{\mathbb{D}}{\overline{\mathcal{E}}}$, is given by
\[\LEn{\E}{x}(z,\lambda):=\Psi(z,x(\lambda))=\frac{x_3(\lambda)z-x_1(\lambda)}{x_2(\lambda)z-1}, \;z,\lambda\in\mathbb{D}.\]
\end{definition}

\begin{proposition} \label{x-S2lf}
Let $\varphi\in\cal{S}_2^{\mathrm{lf}}$. Suppose 
 functions  $x=(x_1,x_2,x_3), y=(y_1,y_2,y_3) \in\hol{\mathbb{D}}{\overline{\mathcal{E}}}$ are
such that 
\[\varphi(z,\lambda)=\Psi(z,x(\lambda))\] 
and 
\[\varphi(z,\lambda)=\Psi(z,y(\lambda))\]
for all $z,\lambda\in\mathbb{D}$. 
Then the following relations hold:

{\em (i)} if $x_1 x_2 \neq x_3$, then $x =y$ on $\D$.

{\em (ii)} if  $x_1 x_2 = x_3$, then  
$y=\left(x_1, y_2, x_1 y_2\right) \; \text{on } \; \D.$
\end{proposition}

\begin{proof} By assumption,
\[\Psi(z,x(\lambda))=\varphi(z,\lambda)=\Psi(z,y(\lambda))\] for all $z, \lambda \in\mathbb{D}$.
Hence \[\frac{x_3(\lambda)z-x_1(\lambda)}{x_2(\lambda)z-1}=\frac{y_3(\lambda)z-y_1(\lambda)}{y_2(\lambda)z-1},\]
 and so
\begin{align*}
x_3(\lambda)y_2(\lambda)z^2-(x_1(\lambda)y_2(\lambda)+x_3(\lambda))z+x_1(\lambda)=&~\\
\;y_3(\lambda)x_2(\lambda)z^2-
(y_1(\lambda)x_2(\lambda)+y_3(\lambda))z+y_1(\lambda) &
\end{align*} for all $z, \lambda \in\mathbb{D}$. 
Therefore $x_1=y_1$, $x_3 y_2=y_3 x_2$ , and 
$x_1y_2+x_3=y_1x_2+y_3$ on $\D$.
Hence, for all $\lambda \in \D$,
\begin{equation} \label{x1x2x3}
(x_3(\lambda)-x_1(\lambda)x_2(\lambda))y_2(\lambda)=(x_3(\lambda)-x_1(\lambda)x_2(\lambda))x_2(\lambda).
\end{equation}

(i) Suppose that $x_1x_2\neq x_3$. 
Since $x_3 -x_1 x_2$ is a nonzero analytic function on $\D$,
 the zeros of this function are isolated  in $\D$. Thus, by
\eqref{x1x2x3},
 $y_2=x_2$ and $y_3=x_3$ on $\D$. Hence $x=y$.

(ii) If $x_1x_2= x_3$, then we have $x_1=y_1$,
$y_3=x_1 y_2$,  and so $y=\left(x_1, y_2, x_1 y_2\right)$ on $\D$.
\end{proof}

One can use Proposition \ref{varphiupsilon} to define 
 the map $\LWfn{\E}$ below.

\begin{definition}\label{LWf-E}
The map $\LWfn{\E}:\cal{S}_2^{\mathrm{lf}}\to\hol{\mathbb{D}}{\overline{\mathcal{E}}}$ is given by the following procedure:

{\em (i)} for $\varphi\in\cal{S}_2^{\mathrm{lf}}$, where $\varphi(z,\lambda)=\frac{a(\lambda)z+b(\lambda)}{c(\lambda)z+1}$, $z,\lambda\in\mathbb{D}$, and $a \neq bc$, 
\[\LWn{\E}{\varphi}=
\left(b,-c,-a\right).\]

{\em (ii)} for $\varphi\in\cal{S}_2^{\mathrm{lf}}$ such that 
$a = bc$, and so $\varphi(z,\lambda)=b(\lambda)$, $z,\lambda\in\mathbb{D}$,  $\LWfn{\E}$ is the set map
\[\LWn{\E}{\varphi}=\{
\left(b,-d,-bd\right), \; \text{where} \; d \;\text{is analytic  and}\; |d|\leq1\; \text{on} \; \D\}.\]
\end{definition}

\begin{proposition} \label{LEINVLWtetra} The following relations hold.

{\em  (i)} for each $x=(x_1,x_2,x_3)\in\hol{\mathbb{D}}{\overline{\mathcal{E}}}$ such that $x_3 \neq x_1 x_2$,
 $$\LWfn{\E}\circ\LEfn{\E}(x)=x.$$

{\em (ii)} for each $\varphi\in\cal{S}_2^{\mathrm{lf}}$ such that $\varphi(z,\lambda)=\frac{a(\lambda)z+b(\lambda)}{c(\lambda)z+1}$, $z,\lambda\in\mathbb{D}$, and $a \neq bc$, 
$$\LEfn{\E}\circ\LWfn{\E}(\varphi)=\varphi.$$
\end{proposition}

\begin{proof}
(i) Let $x=(x_1,x_2,x_3)\in\hol{\mathbb{D}}{\overline{\mathcal{E}}}$ be such that $x_3 \neq x_1 x_2$.
Then 
\[\LEn{\E}{x}=\varphi\in\cal{S}_2^{\mathrm{lf}}, \; \text{ where} \; \varphi(z,\lambda)=\Psi(z,x(\lambda)), \; z,\lambda\in\mathbb{D}. 
\]
Thus \[\varphi(z,\lambda)=\frac{x_3(\lambda)z-x_1(\lambda)}{x_2(\lambda)z-1}
=\frac{-x_3(\lambda)z+x_1(\lambda)}{-x_2(\lambda)z+1}\] for all $z,\lambda\in\mathbb{D}$ and $x_3 \neq x_1 x_2$. By Definition \ref{LWf-E}, 
\[\LWn{\E}{\varphi}=\left(x_1,x_2,x_3\right)=x,\]
and so
$$\LWfn{\E}\circ\LEfn{\E}(x)=x.$$

(ii) Let
 $\varphi\in\cal{S}_2^{\mathrm{lf}}$ be such that $\varphi(z,\lambda)=\frac{a(\lambda)z+b(\lambda)}{c(\lambda)z+1}$, $z,\lambda\in\mathbb{D}$ and $a \neq bc$. Then, by Definition \ref{LWf-E}, 
\[\LWn{\E}{\varphi}=x_{\varphi}=
\left(b,-c,-a\right).\]
Therefore
\[\LEfn{\E}(x_{\varphi})(z,\lambda)=\Psi(z,x_{\varphi}(\lambda))
=\frac{a(\lambda)z+b(\lambda)}{c(\lambda)z+1}=\varphi(z,\lambda)\]
for all $z,\lambda\in\mathbb{D}$.
It follows that $\LEfn{\E}\circ\LWn{\E}{\varphi}=\varphi$ for  $\varphi\in\cal{S}_2^{\mathrm{lf}}$ such that  $a \neq bc$.
\end{proof}


Let us see  how these maps  interact with the other maps in the rich saltire \eqref{rich_str_E}.

\begin{proposition}\label{SELNLEtetra} The following equality
$\SEf\circ\LNfn{\E}=\LEfn{\E}$
holds.
\end{proposition}

\begin{proof}
Let $x=(x_1,x_2,x_3)\in\hol{\mathbb{D}}{\overline{\mathcal{E}}}$.
Then $\LNn{\E}{x}=F\in\cal{S}^{2\times2}$ as defined in Theorem \ref{leftntetrawelldefinedprop} and, by the proof of 
Theorem \ref{leftntetrawelldefinedprop},
\[\SEf(F)(z,\lambda)=\cal{F}_{F(\lambda)}(z)=\Psi(z,x(\lambda))\]
for all $z,\lambda\in\mathbb{D}$.
Hence, by definition,
\[\SEf\circ\LNn{\E}{x}(z,\lambda)=\Psi(z,x(\lambda))=\LEn{\E}{x}(z,\lambda)\]
for all $z,\lambda\in\mathbb{D}$.
It follows that $\SEf\circ\LNn{\E}{x}=\LEn{\E}{x}$ for all $x\in\hol{\mathbb{D}}{\overline{\mathcal{E}}}$ 
and so $\SEf\circ\LNfn{\E}=\LEfn{\E}$.
\end{proof}

\begin{corollary} The following relations hold.
{\em  (i)} For each $x=(x_1,x_2,x_3)\in\hol{\mathbb{D}}{\overline{\mathcal{E}}}$ such that $x_3 \neq x_1 x_2$,
$$\LWfn{\E}\circ\SEf\circ\LNfn{\E}(x)= x.$$

{\em (ii)} for each $\varphi\in\cal{S}_2^{\mathrm{lf}}$ such that $\varphi(z,\lambda)=\frac{a(\lambda)z+b(\lambda)}{c(\lambda)z+1}$, $z,\lambda\in\mathbb{D}$, and $a \neq bc$, 
$$\SEf\circ\LNfn{\E}\circ\LWfn{\E}(\varphi)=\varphi.$$
\end{corollary}

\begin{proof} 
This follows immediately from Proposition \ref{SELNLEtetra} 
and Proposition \ref{LEINVLWtetra}.
\end{proof}

\begin{proposition} The equality
$\LEfn{\E}\circ\LSfn{\E}=\SEf$ holds.
\end{proposition}

\begin{proof} Let $F=\FF\in\cal{S}^{2\times2}$. Then
$\LSn{\E}{F}=(F_{11},F_{22},\det{F})$ and  
\[\LEn{\E}{(F_{11},F_{22},\det{F})}(z,\lambda)=\Psi(z,F_{11}(\lambda),F_{22}(\lambda),\det{F(\lambda)})\]
for all $z,\lambda\in\mathbb{D}$.
Moreover
\begin{align*}\SEf(F)(z,\lambda)=&\cal{F}_{F(\lambda)}(z)\\=&
F_{11}(\lambda)+\frac{F_{12}(\lambda)F_{21}(\lambda)z}{1-F_{22}(\lambda)z}=
\frac{F_{11}(\lambda)-\det{F}(\lambda)z}{1-F_{22}(\lambda)z}
\\=&\Psi(z,F_{11}(\lambda),F_{22}(\lambda),\det{F(\lambda)})\end{align*} 
for all $z,\lambda\in\mathbb{D}$.
It follows that $\LEfn{\E}\circ\LSn{\E}{F}=\SEf(F)$ for all $F\in\cal{S}^{2\times2}$ and so 
$\LEfn{\E}\circ\LSfn{\E}=\SEf$ as required.
\end{proof}


The idea for $\SWfn{\E}$ is that we want to follow Procedure UW with the application of the map $\LSfn{\E}$ to the function produced.
The following proposition will facilitate this.

\begin{proposition}\label{SWTwelldefined}
Let $(N,M)\in\cal{R}_{11}$. Let $\Xi$ be any function constructed from $(N,M)$ by Procedure UW (Theorem 
\ref{procUWcons}).
Then
\begin{align*}\{\LSfn{\E}(F):F\in\UW{N,M}\}
=&\left\{\left(\zeta\Xi_{11},\Xi_{22},\zeta\det\Xi\right):\zeta\in\mathbb{T}\right\}
\subseteq\hol{\mathbb{D}}{\overline{\mathcal{E}}}.\end{align*}
\end{proposition}

\begin{proof} By Proposition \ref{UWfwelldefined},
a function $F=\begin{bmatrix}\zeta_1&0\\0&\zeta_2\end{bmatrix}\Xi
\begin{bmatrix}1&0\\0&\overline{\zeta_2}\end{bmatrix}\in\UW{N,M}$, where $\zeta_1,\zeta_2\in\mathbb{T}$. Thus
\begin{align*}\LSn{\E}{F}=&\left(\zeta_1\Xi_{11},\Xi_{22},\det{\begin{bmatrix}\zeta_1&0\\0&\zeta_2\end{bmatrix}\Xi
\begin{bmatrix}1&0\\0&\overline{\zeta_2}\end{bmatrix}}\right)
=\left(\zeta_1\Xi_{11},\Xi_{22},\zeta_1\det{\Xi}\right).\end{align*}
\end{proof}

\begin{definition} Let $\SWfn{\E}$  be the set-valued map from $\cal{R}_{11}$ to $\hol{\mathbb{D}}{\overline{\mathcal{E}}}$ such that 
\[\SWn{\E}{N,M}=\left\{\left(\zeta\Xi_{11},\Xi_{22},\zeta\det\Xi\right):\zeta\in\mathbb{T}\right\}\]
for all $(N,M)\in\cal{R}_{11}$, where $\Xi=\begin{bmatrix}\Xi_{11}&\Xi_{12}\\\Xi_{21}&\Xi_{22}\end{bmatrix}\in\cal{S}^{2\times2}$ 
is a function constructed from $(N,M)$ by Procedure UW.
\end{definition}

 By Proposition \ref{UWfwelldefined}, $\SWfn{\E}$ is independent of choice of $\Xi$ in $\UW{N,M}$.


\section{A criterion for the solvability of
the  $\mu_{\mathrm{Diag}}$-synthesis
 problem }\label{crit_tetr}

\begin{theorem} \label{criterion_tetr}
Let $\lambda_1,\dots,\lambda_n$ be distinct points in $\mathbb{D}$ and let 
$(x_{1j},x_{2j},x_{3j})\in\overline{\mathcal{E}}$ be such that $x_{1j}x_{2j}\neq x_{3j}$ for $j=1,\dots,n$. Then the following are equivalent.
\begin{enumerate}[\em  (i)]
\item There exists a holomorphic function $x:\mathbb{D}\to\overline{\mathcal{E}}$ such that
\beq\label{inq1.2}
\qquad x(\lambda_j)=(x_{1j},x_{2j},x_{3j})\text{ for }j=1,\dots,n.
\eeq
\item There exists a rational $\overline{\mathcal{E}}$-inner function $x$ 
such that
\beq\label{inq1.3}
\qquad x(\lambda_j)=(x_{1j},x_{2j},x_{3j})\text{ for }j=1,\dots,n.
\eeq
\item For every triple of distinct points   $z_1,z_2, z_3$ in $\D$, there exist positive $3n$-square matrices $N=[N_{il,jk}]^{n,3}_{i,j=1,l,k=1}$ of rank at most $1$,
and $M=[M_{il,jk}]^{n,3}_{i,j=1,l,k=1}$ such that, for $1\leq i,j\leq n$ and $1\leq l,k\leq 3$,
\beq\label{inq1.4}
\qquad 1-\overline{\frac{z_lx_{3i}-x_{1i}}{x_{2i}z_l-1}}\frac{z_kx_{3j}-x_{1j}}{x_{2j}z_k-1}=(1-\overline{z_l}z_k)N_{il,jk}+
(1-\overline{\lambda_i}\lambda_j)M_{il,jk}.
\eeq
\item For some distinct points  $z_1,z_2, z_3$ in $\D$, there exist positive $3n$-square matrices $N=[N_{il,jk}]^{n,3}_{i,j=1,l,k=1}$ of rank at most $1$,
and $M=[M_{il,jk}]^{n,3}_{i,j=1,l,k=1}$ such that
\beq\label{inq1.5}
\qquad \left[1-\overline{\frac{z_lx_{3i}-x_{1i}}{x_{2i}z_l-1}}\frac{z_kx_{3j}-x_{1j}}{x_{2j}z_k-1}\right] \geq
\left[(1-\overline{z_l}z_k)N_{il,jk}\right]+
\left[(1-\overline{\lambda_i}\lambda_j)M_{il,jk}\right].
\eeq
\end{enumerate}
\end{theorem}

\begin{proof} 
Clearly (ii)$\implies$(i) and (iii)$\implies$(iv). We will show that (iii)$\implies$(ii), (iv)$\implies$(i) and (i)$\implies$(iii) to
complete the proof.

(iii)$\implies$(ii):
Suppose that (iii) holds. Then since $N$ is positive and has rank $1$ there are $\gamma_{jk}\in\mathbb{C}$ such that 
for all $j=1,\dots,n$ and 
$k=1,2,3$ \[N_{il,jk}=\overline{\gamma_{il}}\gamma_{jk}.\] 
Similarly since $M$ is positive there is a Hilbert space $H$ 
of dimension at most $3n$ 
and vectors $v_{jk}\in H$ such that for all $j=1,\dots,n$ and 
$k=1,2,3$ \[M_{il,jk}=\langle v_{jk},v_{il}\rangle_H.\] 
Now recall that $\Psi(z_k,x_{1j},x_{2j},x_{3j})=\frac{z_kx_{3j}-x_{1j}}{x_{2j}z_k-1}$. 
Then, as in the proof of Theorem \ref{procUWcons},
we can show that the Gramian of the vectors 
\[\begin{pmatrix}\Psi(z_k,x_{1j},x_{2j},x_{3j})\\\gamma_{jk}\\v_{jk}\end{pmatrix}\in\mathbb{C}^2\oplus H\] 
for all $j=1,\dots,n$ and $k=1,2,3$,
is equal to the Gramian of the vectors \[\begin{pmatrix}1\\z_k\gamma_{jk}\\\lambda_jv_{jk}\end{pmatrix}\in\mathbb{C}^2\oplus H\]
for all $j=1,\dots,n$ and $k=1,2,3$.
Hence  there is a unitary operator $L$ 
on $\mathbb{C}^2\oplus H$ which maps the
vectors \[\begin{pmatrix}\Psi(z_k,x_{1j},x_{2j},x_{3j})\\\gamma_{jk}\\v_{jk}\end{pmatrix}\text{ to the vectors } 
\begin{pmatrix}1\\z_k\gamma_{jk}\\\lambda_jv_{jk}\end{pmatrix}\]
for $j=1,\dots,n$ and $k=1,2,3$. 
Write $L$ as a block operator matrix  $$L=\begin{bmatrix} A & B\\ C& D\end{bmatrix},$$
where $A, D$ act on $\mathbb{C}^2$, $H$ respectively. Then,
for  
$j=1,\dots,n$ and $k=1,2,3$, we obtain the following equations
\[\begin{pmatrix}\Psi(z_k,x_{1j},x_{2j},x_{3j})\\\gamma_{jk}\end{pmatrix}=A\begin{pmatrix}1\\z_k\gamma_{jk}\end{pmatrix}+B\lambda_jv_{jk}\text{ and }
v_{jk}=C\begin{pmatrix}1\\z_k\gamma_{jk}\end{pmatrix}+D\lambda_jv_{jk}.\]
From the second of these equations,
\[v_{jk}=(I-D\lambda_j)^{-1}C\begin{pmatrix}1\\z_k\gamma_{jk}\end{pmatrix},\] 
and so 
\[\begin{pmatrix}\Psi(z_k,x_{1j},x_{2j},x_{3j})\\\gamma_{jk}\end{pmatrix}=
(A+B\lambda_j(I-D\lambda_j)^{-1}C)\begin{pmatrix}1\\z_k\gamma_{jk}\end{pmatrix},\]
for all 
$j=1,\dots,n$ and $k=1,2,3$.
Let 
$\Theta(\lambda)=A+B\lambda(I-D\lambda)^{-1}C=\begin{bmatrix}a(\lambda) & b(\lambda)\\ c(\lambda) & d (\lambda)\end{bmatrix}$.
Since $L$ is unitary and $H$ is finite-dimensional, $\Theta$ is a rational $2\times2$ inner function. 
Hence the function $x:=(a,d,\det \Theta)$ is a rational $\overline{\mathcal{E}}$-inner function.

We claim that $x$ satisfies the interpolation conditions \eqref{inq1.3}
$x(\lambda_j)=(x_{1j},x_{2j},x_{3j})$ for all $j=1,\dots,n$.

 From above 
\[\begin{pmatrix}\Psi(z_k,x_{1j},x_{2j},x_{3j})\\\gamma_{jk}\end{pmatrix}=
\Theta(\lambda_j)\begin{pmatrix}1\\z_k\gamma_{jk}\end{pmatrix}=
\begin{pmatrix}a(\lambda_j)+b(\lambda_j)z_k\gamma_{jk}\\c(\lambda_j)+d(\lambda_j)z_k\gamma_{jk}\end{pmatrix}\]
for $j=1,\dots,n$ and $k=1,2,3$.
Hence 
\[\Psi(z_k,x_{1j},x_{2j},x_{3j})=
a(\lambda_j)+b(\lambda_j)z_k\gamma_{jk}\text{ and }\gamma_{jk}=c(\lambda_j)+d(\lambda_j)z_k\gamma_{jk}\]
and so 
\[\Psi(z_k,x_{1j},x_{2j},x_{3j})=a(\lambda_j)+b(\lambda_j)z_k(1-d(\lambda_j)z_k)^{-1}c(\lambda_j).\]
That is, for each $j=1,\dots,n$, the linear fractional maps
\[\Psi(z_k,x_{1j},x_{2j},x_{3j})=\frac{x_{1j}-x_{3j}z}{1-x_{2j}z}\text{ and }
a(\lambda_j)+\frac{b(\lambda_j)c(\lambda_j)z}{1-d(\lambda_j)z}=
\frac{a(\lambda_j)-(a(\lambda_j)d(\lambda_j)-b(\lambda_j)c(\lambda_j))z}{1-d(\lambda_j)z}\]
agree at three distinct values of $z\in\mathbb{D}$, and so the two maps are the same.
Thus, since $x_{1j}x_{2j}\neq x_{3j}$ for $j=1,\dots,n$,
\[a(\lambda_j)=x_{1j}, d(\lambda_j)=x_{2j}\text{ and } 
\det \Theta(\lambda_j)=a(\lambda_j)d(\lambda_j)-b(\lambda_j)c(\lambda_j)=
x_{3j}.\] 
It follows that $x(\lambda_j)=(x_{1j},x_{2j},x_{3j})$ for $j=1,\dots,n$ 
and so (iii)$\implies$(ii).

(iv)$\implies$(i):
This proof is similar to (iii)$\implies$(ii). The difference is that the Gramian of the vectors 
\[\begin{pmatrix}\Psi(z_k,x_{1j},x_{2j},x_{3j})\\\gamma_{jk}\\v_{jk}\end{pmatrix}\in\mathbb{C}^2\oplus H\]
is less than or equal to the Gramian of the vectors 
\[\begin{pmatrix}1\\z_k\gamma_{jk}\\\lambda_jv_{jk}\end{pmatrix}\in\mathbb{C}^2\oplus H,\] for $j=1,\dots,n$ and $k=1,2,3$. 
Hence there is a contraction $L$ on $\mathbb{C}^2\oplus H$ which
maps the vectors 
\[\begin{pmatrix}\Psi(z_k,x_{1j},x_{2j},x_{3j})\\\gamma_{jk}\\v_{jk}\end{pmatrix}\text{ to the vectors }
\begin{pmatrix}1\\z_k\gamma_{jk}\\\lambda_jv_{jk}\end{pmatrix}.\] 
Since $L$ is a contraction, the map $\Theta$ defined 
by $\Theta(\lambda)=A+B\lambda(I-D\lambda)^{-1}C=\begin{bmatrix}a(\lambda) & b(\lambda)\\ c(\lambda) & d (\lambda)\end{bmatrix}$ 
belongs to $\cal{S}^{2\times2}$ 
and hence $x=(a,d,\det \Theta)\in\hol{\mathbb{D}}{\overline{\mathcal{E}}}$. 
That $x(\lambda_j)=(x_{1j},x_{2j},x_{3j})$ for $j=1,\dots,n$ follows as in the previous part.

(i)$\implies$(iii):
Suppose there is a holomorphic function $x=(x_1,x_2,x_3):\mathbb{D}\to\overline{\mathcal{E}}$ 
satisfying $x(\lambda_j)=(x_{1j},x_{2j},x_{3j})$ for $j=1,\dots,n$. By Theorem  \ref{leftntetrawelldefinedprop}, 
there is a holomorphic function
\[F=\begin{bmatrix}x_1 & f_1\\f_2 &x_2\end{bmatrix}:\mathbb{D}\to\cal{M}_2(\mathbb{C})\] such that $f_2 \neq 0$ and 
$\|F(\lambda)\|\leq1$ for all $\lambda\in\mathbb{D}$ 
and 
\[1-\overline{\Psi(w,x(\mu))}\Psi(z,x(\lambda))=(1-\overline{w}z)\overline{\gamma(\mu,w)}\gamma(\lambda,z)+
(1-\overline{\mu}\lambda)\eta(\mu,w)^*\frac{I-F(\mu)^*F(\lambda)}{1-\overline{\mu}\lambda}\eta(\lambda,z)\]
for all $\mu,\lambda\in\mathbb{D}$ and any $w,z\in\mathbb{C}$ such that $1-x_2(\mu)w\neq0$ and 
$1-x_2(\lambda)z\neq0$, 
where \[\gamma(\lambda,z)=(1-x_2(\lambda)z)^{-1}f_2(\lambda)\text{ and }
\eta(\lambda,z)=\begin{bmatrix}1\\\gamma(\lambda,z)z\end{bmatrix}.\]
Hence for the given $\lambda_j\in\mathbb{D}$, $j=1,\dots,n$, and for all $w,z\in\mathbb{D}$,
\begin{align*}1- & \overline{\Psi(w,x_{1i},x_{2i},x_{3i})}\Psi(z,x_{1j},x_{2j},x_{3j})\\
&=1-\overline{\Psi(w,x(\lambda_i))}\Psi(z,x(\lambda_j))\\
&=(1-\overline{w}z)\overline{\gamma(\lambda_i,w)}\gamma(\lambda_j,z)+
(1-\overline{\lambda_i}\lambda_j)\eta(\lambda_i,w)^*\frac{I-F(\lambda_i)^*F(\lambda_j)}{1-\overline{\lambda_i}\lambda_j}
\eta(\lambda_j,z).\end{align*}
In particular for every triple of distinct points   $z_1,z_2, z_3$ in $\D$,
 and for all $j=1,\dots,n$,
\begin{align*}1- & \overline{\Psi(z_l,x_{1i},x_{2i},x_{3i})}\Psi(z_k,x_{1j},x_{2j},x_{3j})\\
&=(1-\overline{z_l}z_k)\overline{\gamma(\lambda_i,z_l)}\gamma(\lambda_j,z_k)+
(1-\overline{\lambda_i}\lambda_j)\eta(\lambda_i,z_l)^*\frac{I-F(\lambda_i)^*F(\lambda_j)}{1-\overline{\lambda_i}\lambda_j}
\eta(\lambda_j,z_k).\end{align*}
Since $F\in\cal{S}^{2\times2}$ with $f_2 \neq 0$, by Proposition \ref{UEfwelldef},
\[\overline{\gamma(\mu,w)}\gamma(\lambda,z) \text{ and } \eta(\mu,w)^*\frac{I-F(\mu)^*F(\lambda)}{1-\overline{\mu}\lambda}\eta(\lambda,z)\] are kernels on $\mathbb{D}^2$.
Hence the $3n$-square matrices
\[N=[N_{il,jk}]_{i,j=1,l,k=1}^{n,3}:=\left[\overline{\gamma(\lambda_i,z_l)}\gamma(\lambda_j,z_k)\right]_{i,j=1,l,k=1}^{n,3}\]
and
\[M=[M_{il,jk}]_{i,j=1,l,k=1}^{n,3}:=\left[\eta(\lambda_i,z_l)^*\frac{I-F(\lambda_i)^*F(\lambda_j)}{1-\overline{\lambda_i}\lambda_j}
\eta(\lambda_j,z_k)\right]_{i,j=1,l,k=1}^{n,3}\] are positive for all $1\leq i,j\leq n$ and $1\leq l,k\leq3$.
Moreover $N$ is of  rank $1$ and for all
$1\leq i,j\leq n$ and $1\leq l,k\leq3$,
\[1-\overline{\Psi(z_l,x_{1i},x_{2i},x_{3i})}\Psi(z_k,x_{1j},x_{2j},x_{3j})=(1-\overline{z_l}z_k)N_{il,jk}+
(1-\overline{\lambda_i}\lambda_j)M_{il,jk}.\]
It follows that (i)$\implies$(iii).
\end{proof}


\section{Construction of all interpolating functions in $\hol{\mathbb{D}}{\overline{\mathcal{E}}}$.}
\label{Proc_SW}

Theorem \ref{criterion_tetr} gives us a criterion for the solvability of the interpolation problem
\beq\label{GaInterp}
\textit{find } x\in {\mathrm{Hol}}(\D,\overline{\mathcal{E}}) \textit{ such that } x(\la_j)=(x_{1j},x_{2j},x_{3j}) \textit{ for } j=1,\dots,n.
\eeq
The  proof of the theorem contains a description of a process for the derivation of a solution of the problem \eqref{GaInterp} from a feasible pair $(N,M)$ for the inequality  \eqref{inq1.5}  with $\rank N \leq 1$.   The process can be summarized as follows.
\begin{center}  \bf Procedure SW \end{center}
Let $\la_j$ and $(x_{1j},x_{2j},x_{3j})$ be as in Theorem \ref{criterion_tetr}.  Let  $z_1, z_2, z_3 $ be a triple of distint points in $\D$, and $N, M $ be positive $3n$-square matrices such that $\rank N \leq 1$ and the inequality \eqref{inq1.5} holds.
\begin{enumerate}
\item  Choose scalars $\ga_{jk}$ such that $N=\bbm \overline{\ga_{i\ell}}\ga_{jk}\ebm_{i,j=1,\ell,k=1}^{n,3}$.
\item Choose a Hilbert space $\M$ and vectors $v_{jk}\in\M$ such that $M=\bbm\ip{v_{jk}}{v_{i\ell}}_\M\ebm_{i,j=1,\ell,k=1}^{n,3}$.
\item Choose a contraction
\[
\bbm A&B\\C&D \ebm : \C^2\oplus\M \to \C^2\oplus\M
\]
such that
\beq\label{P4.1}
\bbm A&B\\C&D \ebm \bpm 1\\ z_k\ga_{jk} \\ \la_j v_{jk} \epm = \bpm \Psi(z_k,x_{1j},x_{2j},x_{3j}) \\ \ga_{jk} \\ v_{jk} \epm
\eeq
for $j=1,\dots,n$ and $k=1,2,3$.
\item Let
\beq\label{thisish}
x(\la) = \LSn{\E}{A+B\la(I-D\la)\inv C}
\eeq
for $\la\in\D$.
\end{enumerate}
Then $x\in  {\mathrm{Hol}}(\D,\overline{\mathcal{E}})$ and $x(\la_j)=(x_{1j},x_{2j},x_{3j})$ for $j=1,\dots,n$.

The purpose of this section is to show that this procedure in principle yields the {\em general} solution of the problem \eqref{GaInterp}, provided that one can find the general feasible pair $(N,M)$ for the relevant inequality with $\rank N \leq 1$.

\begin{theorem} 
Every solution of an $\overline{\mathcal{E}}$-interpolation problem arises by Procedure SW from a solution 
$(N,M)$ of the corresponding inequality \eqref{inq1.5} with rank of $N$ less than or equal to $1$.
\end{theorem}

\begin{proof} 
Let $\lambda_j,x_{1j},x_{2j},x_{3j}$  be as in Theorem \ref{criterion_tetr} and let 
$x=(x_1,x_2,x_3)\in\hol{\mathbb{D}}{\overline{\mathcal{E}}}$ be 
such that $x(\lambda_j)=(x_{1j},x_{2j},x_{3j})$ for all $j=1,\dots,n$. We must produce a pair of positive matrices $(N,M)$ that satisfy  the inequality \eqref{inq1.5} such that Procedure SW, when applied to $(N,M)$ with appropriate choices, produces $x$.

By Proposition \ref{leftntetrawelldefinedprop} there is a unique 
$F=\FF\in\cal{S}^{2\times2}$ 
such that $F_{11}=x_1$, $F_{22}=x_2$, $\det F=x_3$, $|F_{12}|=|F_{21}|$ a. e. 
on $\mathbb{T}$, $F_{21}$ is outer or $0$ and $F_{12}$ is inner. 
Moreover if 
\[\gamma(\lambda,z)=(1-F_{22}(\lambda)z)^{-1}F_{21}(\lambda)\text{ and }
\eta(\lambda,z)=\begin{bmatrix}1\\z\gamma(z,\lambda)\end{bmatrix}\] 
then
\[1-\overline{\Psi(w,x(\mu))}\Psi(z,x(\lambda))=(1-\overline{w}z)\overline{\gamma(\mu,w)}\gamma(\lambda,z)+
\eta(\mu,w)^*(I-F(\mu)^*F(\lambda))\eta(\lambda,z)\] 
for all $z,\lambda,w,\mu\in\mathbb{D}$.

Since $F\in\cal{S}^{2\times2}$,
\[(\lambda,\mu)\mapsto \frac{I-F(\mu)^*F(\lambda)}{1-\overline{\mu}\lambda}\] is a positive $2\times2$ kernel on $\mathbb{D}$, 
and so there is a Hilbert space $\cal{H}$ and a holomorphic map $U:\mathbb{D}\to\cal{L}(\mathbb{C}^2,\cal{H})$ such that
\[\frac{I-F(\mu)^*F(\lambda)}{1-\overline{\mu}\lambda}=U(\mu)^*U(\lambda)\] for all $\lambda,\mu\in\mathbb{D}$. 
Hence
\[1-\overline{\Psi(w,x(\mu))}\Psi(z,x(\lambda))=(1-\overline{w}z)\overline{\gamma(\mu,w)}\gamma(\lambda,z)+
(1-\overline{\mu}\lambda)\eta(\mu,w)^*U(\mu)^*U(\lambda)\eta(\lambda,z)\] for all $z,\lambda,w,\mu\in\mathbb{D}$.
In particular, for every triple of distinct points $z_1, z_2, z_3$ in $\D$,
\begin{align*}1-&\overline{\Psi(z_l,x_{1i},x_{2i},x_{3i})}\Psi(z_k,x_{1j},x_{2j},x_{3j})\\&=
(1-\overline{z_l}z_k)\overline{\gamma(\lambda_i,z_l)}\gamma(\lambda_j,z_k)+
(1-\overline{\lambda_i}\lambda_j)\langle U(\lambda_j)\eta(z_k,\lambda_j),U(\lambda_i)\eta(z_l,\lambda_i)\rangle_{\cal{H}}\end{align*}
for all $i,j=1,\dots,n$ and $l,k=1,2,3$. 
It follows that the $3n$-square matrices
\[N=\left[\overline{\gamma(z_l,\lambda_i)}\gamma(z_k,\lambda_j)\right]_{i,j=1,l,k=1}^{n,3}\] and
\[M=\left[\langle U(\lambda_j)\eta(z_k,\lambda_j),U(\lambda_i)\eta(z_l,\lambda_i)\rangle_{\cal{H}}\right]_{i,j=1,l,k=1}^{n,3}\]
satisfy the inequality \eqref{inq1.5} and moreover the rank of $N$ is less than or equal to $1$. 
Thus we may apply Procedure SW to $(N,M)$.
In steps (1) and (2) we choose 
$\gamma_{jk}=\gamma(\lambda_j,z_k)$, $\cal{M}=\cal{H}$ and $v_{jk}=U(\lambda_j)\eta(\lambda_j,z_k)$.
As in the proof of Theorem \ref{procUWcons}  
we can show that the Grammian of the vectors 
\[\begin{pmatrix}1\\z\gamma(\lambda,z)\\\lambda U(\lambda)\eta(\lambda,z)\end{pmatrix}\in\mathbb{C}^2\oplus\cal{H}\] 
for all $z,\lambda\in\mathbb{D}$, is equal to the Grammian of the vectors 
\[\begin{pmatrix}\Psi(z,x(\lambda)\\\gamma(\lambda,z)\\U(\lambda)\eta(\lambda,z)\end{pmatrix}\in\mathbb{C}^2\oplus\cal{H}\] 
for all $z,\lambda\in\mathbb{D}$.
Hence there is an isomertry 
\[L_0:\spn\left\{\begin{pmatrix}1\\z\gamma(\lambda,z)\\\lambda U(\lambda)\eta(\lambda,z)\end{pmatrix}:z,\lambda\in\mathbb{D}\right\}\to
\mathbb{C}^2\oplus\cal{H}\] such that
\[L_0\begin{pmatrix}1\\z\gamma(\lambda,z)\\\lambda U(\lambda)\eta(\lambda,z)\end{pmatrix}=
\begin{pmatrix}\Psi(z,x(\lambda)\\\gamma(\lambda,z)\\U(\lambda)\eta(\lambda,z)\end{pmatrix}\] for all $z,\lambda\in\mathbb{D}$.
Now extend $L_0$ to a contraction 
\[L=\begin{bmatrix}A&B\\C&D\end{bmatrix}:\mathbb{C}^2\oplus\cal{H}\to\mathbb{C}^2\oplus\cal{H}.\] 
Then, in particular,
\[L\begin{pmatrix}1\\z_k\gamma(\lambda_j,z_k)\\\lambda_j U(\lambda_j)\eta(\lambda_j,z_k)\end{pmatrix}=
\begin{pmatrix}\Psi(z_k,x(\lambda_j)\\\gamma(\lambda_j,z_k)\\U(\lambda_j)\eta(\lambda_j,z_k)\end{pmatrix}\]
for all $j=1,\dots,n$ and $k=1,2,3$, which is step (3) of Procedure SW.
Hence we can use $L$ in step (4) to obtain a function 
$\widetilde{x}\in\hol{\mathbb{D}}{\overline{\mathcal{E}}}$ such that $\widetilde{x}(\lambda_j)=(x_{1j},x_{2j},x_{3j})$.

We claim that  $\widetilde{x}=x$.
 We already have
\[\begin{pmatrix}\begin{pmatrix}\Psi(z,x(\lambda)\\\gamma(\lambda,z)\end{pmatrix}\\U(\lambda)\eta(\lambda,z)\end{pmatrix}=
L\begin{pmatrix}1\\z\gamma(\lambda,z)\\\lambda U(\lambda)\eta(\lambda,z)\end{pmatrix}=
\begin{pmatrix}A\begin{pmatrix}1\\z\gamma(\lambda,z)\end{pmatrix}+
B\lambda U(\lambda)\eta(\lambda,z)\\C\begin{pmatrix}1\\z\gamma(\lambda,z)\end{pmatrix}+D\lambda U(\lambda)\eta(\lambda,z)
\end{pmatrix}\] and so
\[\begin{pmatrix}\Psi(z,x(\lambda))\\\gamma(\lambda,z)\end{pmatrix}=A\begin{pmatrix}1\\z\gamma(\lambda,z)\end{pmatrix}+
B\lambda U(\lambda)\eta(\lambda,z)\] and
\[(1-D\lambda) U(\lambda)\eta(\lambda,z)=C\begin{pmatrix}1\\z\gamma(\lambda,z)\end{pmatrix}\]
for all $z,\lambda\in\mathbb{D}$.
Hence
\[\begin{pmatrix}\Psi(z,x(\lambda))\\\gamma(\lambda,z)\end{pmatrix}=(A+B\lambda(I-D\lambda)^{-1}C)
\begin{pmatrix}1\\z\gamma(\lambda,z)\end{pmatrix}=
\Theta(\lambda)\begin{pmatrix}1\\z\gamma(\lambda,z)\end{pmatrix}\] and so 
\[\Psi(z,x(\lambda))=
\Theta_{11}(\lambda)+\Theta_{12}(\lambda)z\gamma(\lambda,z)\]
and
\[\gamma(\lambda,z)=
\Theta_{21}(\lambda)+\Theta_{22}(\lambda)z\gamma(\lambda,z)\]
for all $z,\lambda\in\mathbb{D}$. 
It follows that 
\[\Psi(z,x(\lambda))=\Theta_{11}(\lambda)+\frac{\Theta_{12}\Theta_{21}(\lambda)z}{1-\Theta_{22}(\lambda)z}=
\frac{\det\Theta(\lambda)z-\Theta_{11}(\lambda)}{\Theta_{22}(\lambda)z-1}\]
for all $z,\lambda\in\mathbb{D}$, and so, by Proposition \ref{x-S2lf}, 
$\Theta_{11}(\lambda)=x_1(\lambda)$, $\Theta_{22}(\lambda)=x_2(\lambda)$, 
$\det\Theta(\lambda)=x_3(\lambda)$ and 
 $\widetilde{x}=(x_1,x_2,x_3)=x$. 
\end{proof}

The criterion  for  the $\mu_{\mathrm{Diag}}$-synthesis problem (Theorem \ref{NPspectral_tetr}) 
 follows from Theorem \ref{musynthequivfortetra} and Theorem \ref{criterion_tetr}.
The tetrablock $\E$ is a bounded $3$-dimensional domain, which is   more amenable to study than the unbounded $4$-dimensional domain
\[
\Sigma \stackrel{\mathrm{def}}{=} \{ A \in \C^{2\times 2}: \mu_{\mathrm{Diag}}(A) < 1\}.
\]

\begin{theorem} Let $\lambda_1,\dots,\lambda_n$ be distinct points in $\mathbb{D}$ and let 
$(x_{1j},x_{2j},x_{3j})\in\overline{\mathcal{E}}$ be such that $x_{1j}x_{2j}\neq x_{3j}$ 
for $j=1,\dots,n$. 
The $\overline{\mathcal{E}}$-interpolation problem 
\[\lambda_j\in\mathbb{D}\mapsto(x_{1j},x_{2j},x_{3j})\in\overline{\mathcal{E}}\] for $j=1,\dots,n$, is solvable if and only if for some distinct points $z_1, z_2, z_3$ in $\D$, there exist positive $3n$-square
matrices $N=[N_{il,jk}]_{i,j=1,l,k=1}^{n,3}$ of rank $1$ and $M=[M_{il,jk}]_{i,j=1,l,k=1}^{n,3}$ that satisfy 
\begin{equation}\label{LMI}
\left[1-\overline{\frac{z_lx_{3i}-x_{1i}}{x_{2i}z_l-1}}\frac{z_kx_{3j}-x_{1j}}{x_{2j}z_k-1}\right]\geq
\left[(1-\overline{z_l}z_k)N_{il,jk}\right]+
\left[(1-\overline{\lambda_i}\lambda_j)M_{il,jk}\right],
\end{equation} 
\[|N_{il,jk}|\leq\frac{1}{(1-|x_{2i}|)(1-|x_{2j}|)}
\text{ and }
|M_{il,jk}|\leq\frac{2}{|1-\overline{\lambda_i}\lambda_j|}\sqrt{1+\frac{1}{(1-|x_{2i}|)^2}}\sqrt{1+\frac{1}{(1-|x_{2j}|)^2}}.\]
\end{theorem}

\begin{proof}
Sufficiency follows from Theorem \ref{criterion_tetr} (iv)$\implies$(i). To prove necessity, suppose that the interpolation problem is solvable. In the proof of Theorem \ref{criterion_tetr} (i)$\implies$(iii) 
it was shown that, for every triple of distinct points 
$z_1,z_2,z_3$  in $\mathbb{D}$, the inequality \eqref{LMI}
 is satisfied for
\[N=[N_{il,jk}]_{i,j=1,l,k=1}^{n,3}=\left[\overline{\gamma(\lambda_i,z_l)}\gamma(\lambda_j,z_k)\right]_{i,j=1,l,k=1}^{n,3}\]
of rank $1$ and 
\[M=[M_{il,jk}]_{i,j=1,l,k=1}^{n,3}=\left[\eta(\lambda_i,z_l)^*\frac{I-F(\lambda_i)^*F(\lambda_j)}{1-\overline{\lambda_i}\lambda_j}
\eta(\lambda_j,z_k)\right]_{i,j=1,l,k=1}^{n,3}\]
where $\|F(\lambda_j)\|\leq1$ for all $j=1,\dots,n$,
\[\gamma(\lambda_j,z_k)=(1-x_{2j}z_k)^{-1}f_2(\lambda_j) \text{ and }
\eta(\lambda_j,z_k)=\begin{bmatrix}1\\\gamma(\lambda_j,z_k)z_k\end{bmatrix},\]
and $|f_2(\lambda_j)|\leq1$ for all $j=,1,\dots,n$. 
It follows that for all $j=1,\dots,n$ and $k=1,2,3$,
\[|\gamma(\lambda_j,z_k)|\leq\frac{1}{|1-x_{2j}z_k|}\leq\frac{1}{1-|x_{2j}|}\text{ and so }|N_{il,jk}|\leq\frac{1}{(1-|x_{2i}|)(1-|x_{2j}|)}.\]
Moreover for all $j=1,\dots,n$ and $k=1,2,3$,
\[\|\eta(\lambda_j,z_k)\|^2_{\mathbb{C}^2}=
\left|\left|\begin{bmatrix}\gamma(\lambda_j,z_k)z_k\\1\end{bmatrix}\right|\right|^2_{\mathbb{C}^2}=
1+|\gamma(\lambda_j,z_k)z_k|^2\leq1+\frac{1}{(1-|x_{2j}|)^2}\] 
and so 
\begin{align*}|M_{il,jk}|\leq&\frac{\|I-F(\lambda_i)^*F(\lambda_j)\|}{|1-\overline{\lambda_i}\lambda_j|}
\|\eta(\lambda_i,z_l)\|_{\mathbb{C}^2}\|\eta(\lambda_j,z_k)\|_{\mathbb{C}^2}\\\leq&
\frac{2}{|1-\overline{\lambda_i}\lambda_j|}\sqrt{1+\frac{1}{(1-|x_{2i}|)^2}}\sqrt{1+\frac{1}{(1-|x_{2j}|)^2}}.\end{align*}
Thus if the given $\mathcal{E}$-interpolation problem is solvable then there exist positive $3n$-square matrices satisfying the 
required conditions.
\end{proof}


DAVID C. BROWN, School of Mathematics and Statistics, Newcastle University, Newcastle upon Tyne
 NE\textup{1} \textup{7}RU, U.K.~~\\
e-mail\textup{: \texttt{d.c.brown@newcastle.ac.uk}}\\

ZINAIDA A. LYKOVA,
School of Mathematics and Statistics, Newcastle University, Newcastle upon Tyne
 NE\textup{1} \textup{7}RU, U.K.~~\\
e-mail\textup{: \texttt{Zinaida.Lykova@ncl.ac.uk}}\\

N. J. YOUNG, School of Mathematics, Leeds University, Leeds LS2 9JT, U.K.~~
and School of Mathematics and Statistics, Newcastle University, Newcastle upon Tyne
 NE\textup{1} \textup{7}RU, U.K.~~\\
e-mail\textup{: \texttt{N.J.Young@leeds.ac.uk}}\\
\end{document}